\documentclass{ws-aa2}
\usepackage{comment,verbatim}
\usepackage[title,titletoc]{appendix}
\usepackage{tikz}

\DeclareMathOperator{\Bess}{Bess}

\DeclareMathOperator{\im}{Im}

\newtheorem{rhp}[theorem]{RH problem}

\begin{document}
\markboth{Dan Dai, Luming Yao \& Yu Zhai}{Asymptotics of the CHG process with a varying external potential}
\title{Asymptotics of the confluent hypergeometric process with a varying external potential in the super-exponential region}

\author{Dan Dai}

\address{Department of Mathematics, City University of
Hong Kong, Tat Chee Avenue, Kowloon, Hong Kong.\\
dandai@cityu.edu.hk}

\author{Luming Yao\footnote{Corresponding author.}\: $^{\dag \ddag}$}

\address{$^{\dag}$Institute for Advanced Study, Shenzhen University, Shenzhen 518060, China. 
\\$^{\ddag}$School of Mathematical Sciences, Fudan University, Shanghai 200433, China. 
\\ lumingyao@fudan.edu.cn}

\author{Yu Zhai}

\address{Department of Mathematics, City University of
Hong Kong, Tat Chee Avenue, Kowloon, Hong Kong.\\
yuzhai4-c@my.cityu.edu.hk}

\maketitle

\begin{history}
\end{history}

%
%
%

\begin{abstract}
In this paper, we investigate a determinantal point process on the interval $(-s,s)$, associated with the confluent hypergeometric kernel. Let $\mathcal{K}^{(\alpha,\beta)}_s$ denote the trace class integral operator acting on $L^2(-s, s)$ with the confluent hypergeometric kernel. Our focus is on deriving the asymptotics of the Fredholm determinant $\det(I-\gamma \mathcal{K}^{(\alpha,\beta)}_s)$ as $s \to +\infty$, while simultaneously $\gamma \to 1^-$ in a super-exponential region. In this regime of double scaling limit, our asymptotic result also gives us asymptotics of the eigenvalues $\lambda^{(\alpha, \beta)}_k(s)$ of the integral operator $\mathcal{K}^{(\alpha,\beta)}_s$ as $s \to +\infty$.  Based on the integrable structure of the confluent hypergeometric kernel, we derive our asymptotic results by applying the Deift-Zhou nonlinear steepest descent method to analyze the related Riemann-Hilbert problem.
\end{abstract}



\section{Introduction}

Determinantal point processes have attracted significant research interest over the past few decades due to their connections with various topics in both mathematics and physics. These processes are associated with random unitary matrices \cite{And:Gui:Zei,Meh2004}, Dyson Brownian motion \cite{Dyson1962}, free fermionic theory \cite{Dou:Maj:Sch2018}, quantum gravity \cite{Stan:Witt2020}, and many other fields. For more properties and applications of determinantal point processes, one can refer to the comprehensive surveys by Soshnikov \cite{Sosh2000}, Johansson \cite{Johan2006}, Borodin \cite{Borodin-Survey2011}, and references therein.

Let $\mathcal{X}$ be a configuration such that $\#(\mathcal{X} \cap J)$ is finite for any bounded interval $J \subset \mathbb{R}$. A determinantal point process $\mathcal{P}$ is a probability measure on the space of all the configurations, where the $k$-point correlation function $\rho_k(x_1, \ldots, x_k) $ can be expressed in a determinantal form as follows:
\begin{equation} \label{corre-det}
\rho_k(x_1, \ldots, x_k) = \det[K(x_i,x_j)]_{i,j=1}^k.
\end{equation}
In the above formula, $K(\cdot,\cdot)$ is the so-called correlation kernel. One of the central problems in the study of determinantal point processes is about the spacing of the random particles within the process. One well-studied case is the sine point process, characterized by the correlation kernel:
\begin{equation} \label{eq: chg-sine}
K^{\sin} (x,y) = \frac{\sin (x-y)}{\pi (x-y)}.
\end{equation}
The gap probability, i.e. the probability that there is no particle in the interval $(-s,s)$, can be expressed in terms of the following Fredholm determinant:
\begin{equation} \label{eq: F-det-sine}
\det(I-  \mathcal{K}^{\sin}_s),  
\end{equation}
where $\mathcal{K}^{\sin}_s$ is the trace class integral operator acting on $L^2(-s, s)$ with the sine kernel in \eqref{eq: chg-sine}. The large gap asymptotics as $s \to +\infty$ has been extensively studied in the literature. In \cite{Bas:Wid1983,Bud:Bus1995,Dei:Its:Zhou2007,Ehr2006,Kra2004},  it has been shown that:
\begin{align}
\det(I-  \mathcal{K}^{\sin}_s) & =  e^{-\frac{s^2}{2}} s^{-\frac{1}{4}}
     e^{3\zeta'(-1)} 2^{\frac{1}{12}}     \left(1+O(s^{-1})\right), \label{sine-asy-super-exp} \\ 
\det(I- \gamma \mathcal{K}^{\sin}_s) & =  e^{-\frac{2\nu }{\pi}s} (4s)^{\frac{\nu^2}{2 \pi^2}} G^2(1+\frac{i\nu}{2\pi}) G^2(1-\frac{i\nu}{2\pi}) \left[1 + O\left(\frac{1}{s}\right)\right], \  0\leq \gamma < 1.  \label{sine-asy-exp} 
\end{align}
Here, $\nu = - \ln(1- \gamma)$, and $\zeta(\cdot)$ and $G(\cdot)$ represent the Riemann zeta-function and Barnes $G$-function, respectively. 
The above asymptotic results indicate a significant difference: the super-exponential rate $e^{-\frac{s^2}{2}}$ when $\gamma=1$ and the exponential rate $e^{-\frac{2\nu}{\pi}s}$ when $0 \le \gamma < 1$. As $\gamma \to 1^-$, there is a nontrivial transition from the super-exponential region to the exponential region. More precisely, when $\gamma$ belongs to the super-exponential region and approaches 1 rapidly enough, the large gap asymptotics is still dominated by the super-exponential factor $e^{-\frac{s^2}{2}}$. However, when $\gamma$ moves to the exponential region, approaching 1 but at a slower rate, the large gap asymptotics will be governed by the exponential factor $e^{-\frac{2\nu}{\pi}s}$. In the literature, the whole transition picture for the sine point process has been completed in \cite{Both:Dei:Kra2018} and references therein. For the Airy and Bessel point processes, some insights have been provided \cite{Both2016,Both:Buc2018}.

In this paper, we wish to consider a similar problem for the determinantal point process associated with another important kernel, namely the confluent hypergeometric kernel. It is defined as
\begin{equation}\label{chgkernel}
  K^{(\alpha,\beta)}(x,y)=\frac{1}{2\pi i}\frac{\Gamma(1+\alpha+\beta)\Gamma(1+\alpha-\beta)}{\Gamma(1+2\alpha)^2}\frac{\mathbb{A}(x)\mathbb{B}(y)-\mathbb{A}(y)\mathbb{B}(x)}{x-y},
\end{equation}
where $\alpha>-\frac 12$, $\beta \in i \mathbb{R}$ and
\begin{equation}
  \mathbb{A}(x)=\chi_\beta(x)^{\frac{1}{2}}|2x|^{\alpha}e^{-ix}\phi(1+\alpha+\beta,1+2\alpha,2ix), \quad \mathbb{B}(x)=\overline{\mathbb{A}(x)},
\end{equation}
with
\begin{equation}
\chi_\beta(x)=\begin{cases}
e^{\pi i \beta}, & \quad x<0\\
e^{-\pi i \beta}, & \quad x>0
\end{cases}
\end{equation}
and $\phi(a,b,z)$ being the confluent hypergeometric function (cf. \cite[Chap. 13]{NIST})
\begin{equation}
  \phi(a,b,z)=1+\sum_{k=1}^{\infty}\frac{a(a+1)\cdots(a+k-1)z^k}{b(b+1)\cdots(b+k-1)k!}.
\end{equation}
It is straightforward to check that, when $\alpha = \beta = 0$, the kernel $ K^{(0,0)}(x,y) $ reduces to the sine kernel in \eqref{eq: chg-sine}. While, when $\beta = 0$, the kernel \eqref{chgkernel} becomes the type-I Bessel kernel considered in \cite{Kui:Van2003}:
\begin{align}
K^{(\alpha,0)}(x,y) & = K^{\Bess,1} (x,y) = \frac{|x|^\alpha |y|^\alpha }{ x^\alpha y^\alpha}  \frac{\sqrt{xy}}{2}
\frac{J_{\alpha  + \frac{1}{2}}(x) J_{\alpha  - \frac{1}{2}}(y) - J_{\alpha  - \frac{1}{2}}(x) J_{\alpha  + \frac{1}{2}}(y) }{x-y} \label{eq: chg-bessel} 
\end{align}

Similar to \eqref{sine-asy-super-exp} and \eqref{sine-asy-exp}, we are also interested in large gap asymptotics of 
\begin{equation} \label{eq: F-det-deform}
\det(I-\gamma \mathcal{K}^{(\alpha,\beta)}_s), \qquad \gamma \in [0,1],
\end{equation}
where $\mathcal{K}^{(\alpha,\beta)}_s$ is the trace class integral operator acting on $L^2(-s,s)$ with the confluent hypergeometric kernel given in \eqref{chgkernel}. When $\gamma=1$, the large gap asymptotics has been obtained by Deift, Krasovsky and Vasilevska in \cite{Dei:Kra:Vas2011}, as $s \to +\infty$,
\begin{multline} \label{eq:chg-fred-asy}
\det(I- \mathcal{K}^{(\alpha,\beta)}_s) \\
= e^{-\frac{s^2}{2} + 2 \alpha s} s^{-\frac 14 -\alpha^2+\beta^2}  \frac{\sqrt{\pi} G^2(1/2) G(1+2\alpha)}{2^{2\alpha^2} G(1+\alpha+\beta)G(1+\alpha-\beta)} \left[1 + O\left(\frac{1}{s}\right)\right],
\end{multline}
where $G(x)$ is the Barnes $G$-function; see also \cite{Xu:Zhao2020}. When $\gamma$ is a fixed constant in $[0,1)$, \eqref{eq: F-det-deform} can be referred to as the deformed Fredholm determinant, which gives the gap probability that each eigenvalue is independently removed with probability $1-\gamma$. The large gap asymptotics for this deformed Fredholm determinant have recently been derived in \cite{Dai:Zhai2022}, as $s \to +\infty$,
\begin{equation} \label{eq:chg-dfred-asy}
\det(I-\gamma \mathcal{K}^{(\alpha,\beta)}_s) = e^{-\frac{2\nu }{\pi}s} (4s)^{\frac{\nu^2}{2 \pi^2}} e^{\alpha \nu}  G^2(1+\frac{i\nu}{2\pi}) G^2(1-\frac{i\nu}{2\pi}) \left[1 + O\left(\frac{1}{s}\right)\right],
\end{equation}
with $\nu = -  \ln{(1-\gamma)}$. It is worthwhile mentioning that similar large gap asymptotics for undeformed and deformed Fredholm determinants associated with more general kernels have also been derived in the literature, for example, see \cite{Dai:Xu:Zhang2021,Dai:Xu:Zhang2022,Dai:Xu:Zhang2022-2,Yao:Zhang2023}.

In this paper, we will deepen our understanding about the confluent hypergeometric process by deriving the large gap probability of \eqref{eq: F-det-deform} as $s \to +\infty$ and $\gamma \to 1^-$ in the super-exponential region. This asymptotic result will also provide valuable information about asymptotics of the eigenvalues $\lambda^{(\alpha, \beta)}_k(s)$ of the integral operator $\mathcal{K}^{(\alpha,\beta)}_s$ as $s \to +\infty$. 

The rest of this paper is organized as follows. In Section \ref{sec:result}, we state our main results for the large gap asymptotics, along with the asymptotics for eigenvalues $\lambda^{(\alpha, \beta)}_k(s)$. In Section \ref{sec:modelrhp}, we establish a connection between the deformed Fredholm determinant \eqref{eq: F-det-deform} and a model Riemann-Hilbert (RH) problem. Then, in Section \ref{analysistopsi}, we perform the Deift-Zhou steepest descent analysis to study asymptotics of this RH problem. Finally, we present the proof of our main theorem in Section \ref{proof}.

\paragraph{Notations} Throughout this paper, the following notations are frequently used.
\begin{itemize}
\item  The constants $h_k$ are given by
\begin{equation} \label{eq:hn-def}
  \int_{\mathbb{R}}\pi_k(x)\pi_n(x)w(x)dx =
  \begin{cases}
    0, \qquad & k \neq n, \\
    h_k, \qquad & k = n,
  \end{cases}
\end{equation} 
where 
\begin{equation}\label{hermiteweight}
  w(x): =w(x; \alpha, \beta) = \left\{
  \begin{aligned}
    & (-x)^{2\alpha}e^{-x^2},     & x<0,\\
    & e^{-2\pi i \beta}x^{2\alpha}e^{-x^2},     & x>0,
  \end{aligned}
  \right.
\end{equation}
with $\alpha > -\frac{1}{2}$, $\beta \in i \mathbb{R}$, and $\pi_k(x) = x^k + \cdots$ is the corresponding  $k$-th monic orthogonal polynomial.
\item The constants $\gamma_k$ are defined as
\begin{equation}\label{defgammak}
\gamma_k = -\frac{h_k}{2 \pi i}.
\end{equation} 
\item The parameter $t$ is related to $s$ through $t= -4 i s$.
\item $\mu$ is a real number with $\mu \in [-\frac 12, \frac 12).$
\end{itemize}

\section{Statement of results}
\label{sec:result}

Our main result for the large gap asymptotics is given in the following theorem.

\begin{theorem}\label{mainresult_}
Define $p:=p(\chi)$ for $\chi \in \mathbb{R}$ as an integer-valued function such that $p=1$ for $\chi < \frac{1}{2}$ and $ p = \lfloor \chi + \frac{3}{2} \rfloor$ for $\chi \geq \frac{1}{2}$. 
 Assume that, as $s \to +\infty$, $\nu = - \ln(1- \gamma) \to +\infty$ in such a way that
\begin{equation} \label{eq:nu-s-relation}
   \nu \geq 2s-(\chi + \alpha)\ln(4s),
  \end{equation}  
then we have 
  \begin{multline}\label{fdeterminant}
     \det (I - \gamma \mathcal{K}^{(\alpha, \beta)}_s) = e^{-\frac{s^2}{2}+2\alpha s} s^{-\frac 14 -\alpha^2+\beta^2} 
     \frac{\sqrt{\pi}G^2(\frac{1}{2})G(1+2\alpha)}{2^{2\alpha^2}G(1+\alpha+\beta)G(1+\alpha-\beta)} \\
     \times \prod_{k=0}^{p-1}\left(1+ \frac{h_k e^{\pi i \beta}}{2\pi} (4s)^{-\frac{1}{2}-k-\alpha}e^{2s-\nu}\right)
     \left(1+O(s^{-\frac{1}{2}}\ln s)\right),
  \end{multline}
where $\alpha>-\frac 12$, $\beta \in i \mathbb{R}$ and $h_k$ is defined in \eqref{eq:hn-def}.
\end{theorem}

\begin{remark}
One can compare the asymptotics in \eqref{fdeterminant} with that in \eqref{eq:chg-fred-asy}. When $\gamma = 1$, we have $\nu = +\infty$. Consequently, \eqref{fdeterminant} reduces to \eqref{eq:chg-fred-asy}. When $\gamma \to 1^-$ within the region defined by \eqref{eq:nu-s-relation}, only a finite number of additional terms contribute to the asymptotics. However, the primary asymptotic behavior is still dominated by the super-exponential factor $e^{-\frac{s^2}{2}}$. When $\gamma \to 1^-$ at a slower rate, the product in \eqref{fdeterminant} becomes an infinite product. Consequently, the asymptotics change to an exponential type, which is similar to \eqref{eq:chg-dfred-asy}. Describing this transition requires delicate uniform asymptotics, which is beyond the scope of this paper. To the best of my knowledge, the rigorous analysis of this transition has only been established for the sine point process; for more details, refer to \cite{Both:Dei:Kra2018}.
\end{remark}

\begin{remark}
It is interesting to note that the constants $h_k$ appearing in the product in \eqref{fdeterminant} are related to the weight function \eqref{hermiteweight}. The weight function possesses both an algebraic and jump singularity at $x = 0$, which is known as a Fisher-Hartwig singularity. Notably, in one of the earliest papers, the confluent hypergeometric kernel \eqref{chgkernel} was originally studied in the context of a circular unitary ensemble with a Fisher-Hartwig singularity \cite{Dei:Kra:Vas2011}. It is intriguing to observe that this singularity resurges in the current context.
\end{remark}





Recalling the relation \eqref{eq: chg-bessel}  among the confluent hypergeometric kernel, the type-I Bessel kernel, and the sine kernel, one immediately obtain the following large gap asymptotics.

\begin{corollary}
Under the same condition as in Theorem \ref{mainresult_}, we have
  \begin{multline}\label{besselfred}
     \det (I - \gamma \mathcal{K}^{\Bess, 1}_s) = e^{-\frac{s^2}{2}+2\alpha s}s^{-\alpha^2 -\frac{1}{4}}
     \frac{\sqrt{\pi}G(\frac{1}{2})G(1+2\alpha)}{2^{2\alpha^2}G(1+\alpha)^2} \\
     \times \prod_{k=0}^{p-1}\left(1+ \frac{\tilde{h}_k}{2\pi} (4s)^{-\frac{1}{2}-k-\alpha}e^{2s-\nu}\right)
     \left(1+O(s^{-\frac{1}{2}}\ln s)\right),
  \end{multline}
where $\alpha > -\frac 12$ and $\tilde{h}_k$ is defined in \eqref{eq:hn-def} with the weight function replaced by $\tilde{w}(x) = w(x;\alpha, 0)$.
\end{corollary}

\begin{corollary} \label{coro-sine}
Under the same condition as in Theorem \ref{mainresult_}, we have
  \begin{equation}
  \begin{aligned}
     \det (I - \gamma \mathcal{K}^{\sin}_s) &= e^{-\frac{s^2}{2}}s^{-\frac{1}{4}}
     e^{2\zeta'(-1)} 2^{\frac{1}{12}}
     \prod_{k=0}^{p-1}\left(1+ \frac{k!2^{-3k-2}}{\sqrt{\pi}}s^{-\frac{1}{2}-k}e^{2s-\nu}\right)\\
    &\quad \times \left(1+O(s^{-\frac{1}{2}}\ln s)\right),
     \end{aligned}
  \end{equation}
  where $\zeta(\cdot)$ is the Riemann zeta-function.
\end{corollary}

The result in Corollary \ref{coro-sine} agrees with that in \cite[Theorem 1.12]{Both:Dei:Kra2015}.

\subsection*{Application}

The asymptotics of the deformed Fredholm determinant also give us information about the eigenvalues $\{\lambda^{(\alpha, \beta)}_k(s)\}_{k=0}^{\infty}$ of the integral operator $\mathcal{K}^{(\alpha,\beta)}_s$. Let us order them as $ \lambda^{(\alpha, \beta)}_0(s) > \lambda^{(\alpha, \beta)}_1(s) > \cdots$. Using standard operator theory techniques (for example, see \cite{Both2016}), we know that $0<\lambda^{(\alpha, \beta)}_k(s)<1$ and $\lambda^{(\alpha, \beta)}_k(s) \to 1$ as $s \to +\infty$. With Theorem \ref{mainresult_}, we are able to derive more detailed asymptotics for these eigenvalues.

\begin{corollary}
  For any fixed $k \in \mathbb{Z}_{\geq 0}$, we have, as $s \to +\infty$,
  \begin{equation}
    1 - \lambda^{(\alpha, \beta)}_k(s) = \frac{2 \pi}{h_k e^{\pi i \beta}}(4s)^{\frac{1}{2}+k+\alpha}e^{-2s}(1+o(1)),
  \end{equation}
  where $h_k$ is defined in \eqref{eq:hn-def}.
\end{corollary}

\begin{proof}
  First, as $0<\lambda^{(\alpha, \beta)}_k(s)<1$ , we have
  \begin{equation}
    1 + e^{-\nu}\frac{\lambda^{(\alpha, \beta)}_p}{1-\lambda^{(\alpha, \beta)}_p} \geq 1 = \frac{\det (I - \gamma \mathcal{K}^{(\alpha,\beta)}_s)}{\det (I - \gamma \mathcal{K}^{(\alpha,\beta)}_s)}, \qquad \forall p \in \mathbb{Z}_{\geq 0}.
  \end{equation}
Now, let us make use of the main expansion \eqref{fdeterminant} in Theorem \ref{mainresult_} to replace the numerator and the denominator on the right hand side of the above formula. We choose different final index in the product $\prod_{k=1}^{p-1} (\cdots)$, that is, $p' = p+1$ in the numerator and $p'' = p$ in the denominator. After canceling out the terms from $k=0$ to $p-1$ in the product of  the numerator and denominator, the only remaining term in the numerator would be the one corresponding to $k=p$. This gives us

  \begin{equation}
    1 + e^{-\nu}\frac{\lambda^{(\alpha, \beta)}_p}{1-\lambda^{(\alpha, \beta)}_p} \geq \left( 1 + \frac{h_p e^{\pi i \beta}}{2 \pi}(4s)^{-\frac{1}{2}-p-\alpha}e^{2s - \nu} \right)(1+o(1)), \qquad s \to +\infty.
  \end{equation}
  Hence, as $s \to +\infty$, we have
  \begin{equation}\label{eigenvaluegeq}
    \frac{\lambda^{(\alpha, \beta)}_p}{1-\lambda^{(\alpha, \beta)}_p} \geq \frac{h_p e^{\pi i \beta}}{2 \pi}(4s)^{-\frac{1}{2}-p-\alpha}e^{2s}(1+o(1)), \qquad p \in \mathbb{Z}_{\geq 0}.
  \end{equation}

  Next, according to Lidskii's Theorem, we have for $l \in \mathbb{Z}_{\geq 0}$
  \begin{equation}
    \begin{aligned}
    \frac{\det (I - \gamma \mathcal{K}^{(\alpha,\beta)}_s)}{\det (I - \mathcal{K}^{(\alpha,\beta)}_s)} & = \det (I + e^{-\nu}\mathcal{K}^{(\alpha,\beta)}_s (I-\mathcal{K}^{(\alpha,\beta)}_s))^{-1}) \\
    & = \prod_{k=0}^{l-1} \left(1 + e^{-\nu}\frac{\lambda^{(\alpha, \beta)}_k}{1-\lambda^{(\alpha, \beta)}_k} \right) \det(I + e^{-\nu}\mathcal{K}_l (I-\mathcal{K}_l)^{-1}),
    \end{aligned}
  \end{equation}
  where $\mathcal{K}_l = \mathcal{K}^{(\alpha,\beta)}_s \cdot P_l$ and $P_l$ is the projection operator that projects on the space of eigenvectors of $\mathcal{K}^{(\alpha,\beta)}_s$ with the corresponding eigenvalues $\{\lambda^{(\alpha, \beta)}_j: j \geq l\}$. Substituting \eqref{fdeterminant} into the above identity, we have
   \begin{equation}\label{eq:p-l}
  \begin{aligned}
 & \prod_{k=0}^{p-1}\left(1+ \frac{h_k e^{\pi i \beta}}{2\pi} (4s)^{-\frac{1}{2}-k-\alpha}e^{2s-\nu}\right) \left(1+O(s^{-\frac{1}{2}}\ln s)\right)\\
 &=\prod_{k=0}^{l-1} \left(1 + e^{-\nu}\frac{\lambda^{(\alpha, \beta)}_k}{1-\lambda^{(\alpha, \beta)}_k} \right) \det(I + e^{-\nu}\mathcal{K}_l (I-\mathcal{K}_l)^{-1}).
\end{aligned}
 \end{equation}
Then, we begin by selecting $\chi=\frac{1}{2}$, which is equivalent to setting $p=2$, and taking $l=1$. Employing the established fact that $\det(I + e^{-\nu}\mathcal{K}_l (I-\mathcal{K}_l)^{-1}) \geq 1$, we derive 
 \begin{equation}
 \left(1 + e^{-\nu}\frac{\lambda^{(\alpha, \beta)}_0}{1-\lambda^{(\alpha, \beta)}_0} \right) \le   \left(1+ \frac{h_0 e^{\pi i \beta}}{2\pi} (4s)^{-\frac{1}{2}-\alpha}e^{2s-\nu}\right) (1+o(1)),
 \end{equation} 
which gives us
  \begin{equation}
    \frac{\lambda^{(\alpha, \beta)}_0}{1-\lambda^{(\alpha, \beta)}_0} \leq  \frac{h_0 e^{\pi i \beta}}{2 \pi}(4s)^{-\frac{1}{2}-\alpha}e^{2s} (1+o(1)).
  \end{equation}
  Combined with \eqref{eigenvaluegeq}, we have
  \begin{equation}\label{eq:lambda0}
    \frac{\lambda^{(\alpha, \beta)}_0}{1-\lambda^{(\alpha, \beta)}_0}  =  \frac{h_0 e^{\pi i \beta}}{2 \pi}(4s)^{-\frac{1}{2}-\alpha}e^{2s} (1+o(1)), \qquad s \to +\infty.
  \end{equation}
Next, we select $\chi=\frac{3}{2}$, i.e. we take $p=3$ and $l=2$ in \eqref{eq:p-l}, we have
 \begin{equation}
  \begin{aligned}
 & \left(1+ \frac{h_0 e^{\pi i \beta}}{2\pi} (4s)^{-\frac{1}{2}-\alpha}e^{2s-\nu}\right) \left(1+ \frac{h_1 e^{\pi i \beta}}{2\pi} (4s)^{-\frac{3}{2}-\alpha}e^{2s-\nu}\right)\left(1+O(s^{-\frac{1}{2}}\ln s)\right)\\
 &= \left(1 + e^{-\nu}\frac{\lambda^{(\alpha, \beta)}_0}{1-\lambda^{(\alpha, \beta)}_0} \right) \left(1 + e^{-\nu}\frac{\lambda^{(\alpha, \beta)}_1}{1-\lambda^{(\alpha, \beta)}_1} \right) \det(I + e^{-\nu}\mathcal{K}_2 (I-\mathcal{K}_2)^{-1}).
\end{aligned}
 \end{equation}
Applying \eqref{eq:lambda0} and the fact that $\det(I + e^{-\nu}\mathcal{K}_l (I-\mathcal{K}_l)^{-1}) \geq 1$ gives us
 \begin{equation}
    \frac{\lambda^{(\alpha, \beta)}_1}{1-\lambda^{(\alpha, \beta)}_1}  =  \frac{h_1 e^{\pi i \beta}}{2 \pi}(4s)^{-\frac{3}{2}-\alpha}e^{2s} (1+o(1)), \qquad s \to +\infty.
  \end{equation}
  Then, iterating this approach for general $k > 0$, we have
  \begin{equation}
    \frac{\lambda^{(\alpha, \beta)}_k}{1-\lambda^{(\alpha, \beta)}_k}  =  \frac{h_k e^{\pi i \beta}}{2 \pi}(4s)^{-\frac{1}{2}-k-\alpha}e^{2s} (1+o(1)), \qquad s \to +\infty.
  \end{equation}
  This completes our proof.
\end{proof}

Following a similar derivation, we have the asymptotics for the individual eigenvalues of the trace class operators with respect to the type-I Bessel kernel and the sine kernel immediately.

\begin{corollary}
  Let $\{\lambda^{\Bess,1}_k(s)\}_{k=0}^{\infty}$ and $\{\lambda^{\sin}_k(s)\}_{k=0}^{\infty}$ denote the eigenvalues of the integral operator associated with the type-I Bessel kernel and the sine kernel, respectively. Then, as $s \to +\infty$, we have
  \begin{align}
    1-\lambda^{\Bess,1}_k(s) &= \frac{2\pi}{\tilde{h}_k}(4s)^{\frac{1}{2}+k+\alpha}e^{-2s}(1+o(1)),\label{besseigenvalue} \\
    1-\lambda^{\sin}_k(s) &= \frac{\sqrt{\pi}2^{3k+2}}{k!}s^{\frac{1}{2}+k}e^{-2s}(1+o(1)),
  \end{align}
  where $\tilde{h}_k$ is defined in \eqref{eq:hn-def} with the weight function replaced by $\tilde{w}(x) = w(x;\alpha, 0)$.
\end{corollary}

\section{Model RH problem}\label{sec:modelrhp}

Previous work in \cite{Dai:Zhai2022, Xu:Zhao2020} revealed the importance of the model RH problem in the analysis of the Fredholm determinant. In this section, we reconstruct a model RH problem for $\Psi(z) = \Psi(z;t)$ with $t \in -i(0,+\infty)$ as follows.
\begin{rhp}\label{modelrhp}
 \hfill
 \begin{itemize}
  \item [(a)] $\Psi(z;t)$ is analytic for $z \in \mathbb{C} \setminus \{\cup_{i=1}^7 \Sigma_i \}$, where the oriented contours are defined as
        \begin{equation*}
         \begin{aligned}
           & \Sigma_1=1+e^{\frac{\pi i}{4}}\mathbb{R}^{+}, \qquad \Sigma_2=-1+e^{\frac{3\pi i}{4}}\mathbb{R}^{+}, \qquad
          \Sigma_3=-1+e^{-\frac{3\pi i}{4}}\mathbb{R}^{+}, \qquad     \\
           & \Sigma_4=e^{-\frac{\pi i}{2}}\mathbb{R}^{+}, \qquad\Sigma_5=1+e^{-\frac{\pi i}{4}}\mathbb{R}^{+}, \qquad \Sigma_6=(0,1), \qquad \Sigma_7=(-1,0);
         \end{aligned}
       \end{equation*}
        see Figure \ref{figure1}.

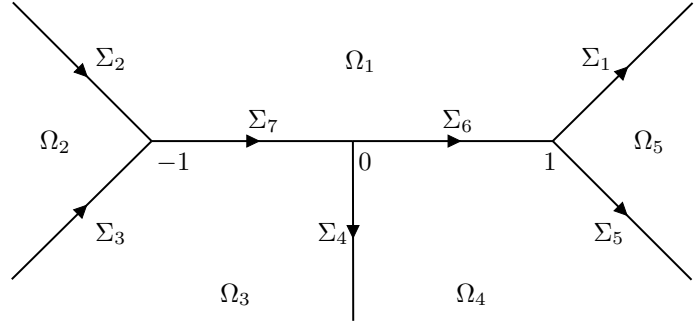
\begin{figure}[t]
	\center
\tikzset{every picture/.style={line width=0.75pt}} 

\begin{tikzpicture}[x=0.75pt,y=0.75pt,yscale=-0.7,xscale=0.7]

\draw    (162.76,129.4) -- (307.76,129.4) ;
\draw [shift={(240.26,129.4)}, rotate = 180] [fill={rgb, 255:red, 0; green, 0; blue, 0 }  ][line width=0.08]  [draw opacity=0] (8.93,-4.29) -- (0,0) -- (8.93,4.29) -- cycle    ;
\draw    (62.76,29.4) -- (162.76,129.4) ;
\draw [shift={(116.29,82.94)}, rotate = 225] [fill={rgb, 255:red, 0; green, 0; blue, 0 }  ][line width=0.08]  [draw opacity=0] (8.93,-4.29) -- (0,0) -- (8.93,4.29) -- cycle    ;
\draw    (451.76,129.4) -- (551.76,229.4) ;
\draw [shift={(505.29,182.94)}, rotate = 225] [fill={rgb, 255:red, 0; green, 0; blue, 0 }  ][line width=0.08]  [draw opacity=0] (8.93,-4.29) -- (0,0) -- (8.93,4.29) -- cycle    ;
\draw    (62,229) -- (162.76,129.4) ;
\draw [shift={(115.93,175.69)}, rotate = 135.33] [fill={rgb, 255:red, 0; green, 0; blue, 0 }  ][line width=0.08]  [draw opacity=0] (8.93,-4.29) -- (0,0) -- (8.93,4.29) -- cycle    ;
\draw    (451.76,129.4) -- (552.51,29.8) ;
\draw [shift={(505.69,76.09)}, rotate = 135.33] [fill={rgb, 255:red, 0; green, 0; blue, 0 }  ][line width=0.08]  [draw opacity=0] (8.93,-4.29) -- (0,0) -- (8.93,4.29) -- cycle    ;
\draw    (307.76,129.4) -- (452.76,129.4) ;
\draw [shift={(385.26,129.4)}, rotate = 180] [fill={rgb, 255:red, 0; green, 0; blue, 0 }  ][line width=0.08]  [draw opacity=0] (8.93,-4.29) -- (0,0) -- (8.93,4.29) -- cycle    ;
\draw    (307.76,129.4) -- (308,259) ;
\draw [shift={(307.89,199.2)}, rotate = 269.89] [fill={rgb, 255:red, 0; green, 0; blue, 0 }  ][line width=0.08]  [draw opacity=0] (8.93,-4.29) -- (0,0) -- (8.93,4.29) -- cycle    ;

\draw (309.76,136) node [anchor=north west][inner sep=0.75pt]   [align=left] {0};
\draw (443,136) node [anchor=north west][inner sep=0.75pt]   [align=left] {1};
\draw (164,136) node [anchor=north west][inner sep=0.75pt]   [align=left] {$-1$};
\draw (470,60) node [anchor=north west][inner sep=0.75pt]   [align=left] {$\Sigma_1$};
\draw (120,60) node [anchor=north west][inner sep=0.75pt]   [align=left] {$\Sigma_2$};
\draw (120,186) node [anchor=north west][inner sep=0.75pt]   [align=left] {$\Sigma_3$};
\draw (280,186) node [anchor=north west][inner sep=0.75pt]   [align=left] {$\Sigma_4$};
\draw (479,186) node [anchor=north west][inner sep=0.75pt]   [align=left] {$\Sigma_5$};
\draw (370,105) node [anchor=north west][inner sep=0.75pt]   [align=left] {$\Sigma_6$};
\draw (230,105) node [anchor=north west][inner sep=0.75pt]   [align=left] {$\Sigma_7$};
\draw (300,62) node [anchor=north west][inner sep=0.75pt]   [align=left] {$\Omega_1$};
\draw (80,120) node [anchor=north west][inner sep=0.75pt]   [align=left] {$\Omega_2$};
\draw (210,230) node [anchor=north west][inner sep=0.75pt]   [align=left] {$\Omega_3$};
\draw (380,230) node [anchor=north west][inner sep=0.75pt]   [align=left] {$\Omega_4$};
\draw (508,120) node [anchor=north west][inner sep=0.75pt]   [align=left] {$\Omega_5$};

\end{tikzpicture}        
 \caption{Contours for the model RH problem. Regions $\Omega_i$, $i =1, \cdots, 5,$ are also depicted.}
         \label{figure1}
        \end{figure}        
        

  \item [(b)] $\Psi$ has limiting values $\Psi_{\pm}(z;t)$ for $z \in \cup_{i=1}^7 \Sigma_i$, where $\Psi_+$ and $\Psi_-$ denote the values of $\Psi$ taken from the left and right side of $\Sigma_i$, respectively. Moreover, they satisfy the following jump conditions
        \begin{equation}\label{jumpforpsi}
         \Psi_+(z;t)=\Psi_-(z;t) \left\{
         \begin{aligned}
           & \begin{pmatrix} 1 & 0 \\ e^{-\pi i(\alpha-\beta)} & 1 \end{pmatrix}, & z \in \Sigma_1, \\
           & \begin{pmatrix} 1 & 0 \\ e^{\pi i(\alpha-\beta)} & 1 \end{pmatrix}, & z \in \Sigma_2, \\
           & \begin{pmatrix} 1 & -e^{-\pi i(\alpha-\beta)} \\ 0 & 1 \end{pmatrix}, & z \in \Sigma_3, \\
           & e^{2\pi i \beta \sigma_3}, & z \in \Sigma_4, \\
           & \begin{pmatrix} 1 & -e^{\pi i(\alpha-\beta)} \\ 0 & 1 \end{pmatrix}, & z \in \Sigma_5, \\
           & \begin{pmatrix} 0 & -e^{\pi i(\alpha-\beta)} \\ e^{-\pi i(\alpha-\beta)} & 1-\gamma \end{pmatrix}, & z \in \Sigma_6, \\
           & \begin{pmatrix} 0 & -e^{-\pi i(\alpha-\beta)} \\ e^{\pi i(\alpha-\beta)} & 1-\gamma \end{pmatrix}, & z \in \Sigma_7. \\
         \end{aligned}
         \right.
        \end{equation}
        For convenience, we denote the jump matrices on $\Sigma_i$ as $J_i$.

  \item [(c)] As $z \to \infty$, we have
        \begin{equation}\label{infinitybehaviorforpsi}
         \Psi(z;t)=\left(I + \frac{\Psi_1(t)}{z} + \frac{\Psi_2(t)}{z^2} + O\left(\frac{1}{z^3}\right)\right)z^{-\beta\sigma_3}e^{\frac{tz}{4}\sigma_3},
        \end{equation}
        where the branch cut of $z^\beta$ is taken along  the negative imaginary axis such that $\arg z \in (-\frac{\pi}{2}, \frac{3\pi}{2})$ and $\sigma_3$ is the third Pauli matrix $\begin{pmatrix} 1 & 0 \\ 0 & -1 \end{pmatrix}$.

  \item [(d)] As $z \to 0$ and  $z \in \Omega_i, \, i=1,3,4$, we have
  \begin{equation}\label{localbehaviorforpsinear0}
         \Psi(z;t)=\Psi^{(0)}(z;t)z^{\alpha \sigma_3} \begin{cases}
         \begin{pmatrix} 1 & (1-\gamma) \frac{\sin(\alpha + \beta)\pi}{\sin 2\alpha \pi}) \\ 0 & 1 \end{pmatrix}C_i^{(0)}, & \textrm{if } 2\alpha \notin \mathbb{N}, \vspace{5pt} \\
         \begin{pmatrix} 1 & \frac{(-1)^{2\alpha}(1-\gamma)}{\pi} \sin(\alpha + \beta)\pi \ln z \\ 0 & 1 \end{pmatrix}C_i^{(0)}, & \textrm{if } 2\alpha \in \mathbb{N},
         \end{cases}
        \end{equation}
        where both $z^\alpha$ and $\ln z$ take the principal branch with $\arg z \in(-\pi, \pi)$,  and the constant matrices $C_i^{(0)}$ are given by
        \begin{equation}
         C_1^{(0)}=I, \qquad  C_3^{(0)}=J_6^{-1}J_4^{-1}, \qquad  C_4^{(0)}=J_6^{-1}
        \end{equation}
  with jump matrices $J_i$ given in \eqref{jumpforpsi}. Here, $\Psi^{(0)}(z;t)=\Psi^{(0)}_0(t)\left(I+\Psi^{(0)}_1(t)z+O(z^2)\right)$ is analytic at $z=0$.

  \item [(e)] As $z \to 1$ and $z \in \Omega_i, \, i=1,4,5$, we have
        \begin{equation}\label{localbehaviorforpsinear1}
         \Psi(z;t)=\Psi^{(1)}(z;t)
         \begin{pmatrix} 1 & -\frac{\gamma e^{\pi i (\alpha-\beta)}}{2\pi i}\ln(z-1) \\ 0 & 1 \end{pmatrix} C_i^{(1)},
        \end{equation}
        where $\ln (z-1)$ takes the principal branch with $\arg (z-1) \in (-\pi,\pi)$, and the constant matrices $C_i^{(1)}$ are given by
        \begin{equation}
         C_1^{(1)}=I, \qquad  C_4^{(1)}=J_1^{-1}J_5^{-1}, \qquad  C_5^{(1)}=J_1^{-1}
        \end{equation}
        Here, $\Psi^{(1)}(z;t)=\Psi^{(1)}_0(t)\left(I+\Psi^{(1)}_1(t)(z-1)+O(z-1)^2)\right)$ is analytic at $z=1$.

  \item [(f)] As $z \to -1$ and $z \in \Omega_i, \, i=1,2,3$, we have
        \begin{equation}\label{localbehaviorforpsinear-1}
         \Psi(z;t)=\Psi^{(-1)}(z;t)
         \begin{pmatrix} 1 & \frac{\gamma e^{-\pi i (\alpha-\beta)}}{2\pi i}\ln(z+1) \\ 0 & 1 \end{pmatrix} C_i^{(-1)},
        \end{equation}
        where the branch cut of $\ln (z+1)$ is taken along $(-1, \infty)$ such that $\arg (z+1) \in(0,2\pi)$, and the constant matrices $C_i^{(-1)}$ are given by
        \begin{equation}
         C_1^{(-1)}=I, \qquad  C_2^{(-1)}=J_2^{-1}, \qquad  C_3^{(-1)}=J_2^{-1}J_3^{-1}.
        \end{equation}
        Here, $\Psi^{(-1)}(z;t)=\Psi^{(-1)}_0(t)\left(I+\Psi^{(-1)}_1(t)(z+1)+O(z+1)^2)\right)$ is analytic at $z=-1$.
 \end{itemize}
\end{rhp}

The following proposition has already been demonstrated in \cite{Dai:Zhai2022}, which gives the existence and uniqueness of the solution to the model RH problem for $\Psi(z;t)$.

\begin{proposition}
For $\alpha > -\frac{1}{2}$, $\beta \in i \mathbb{R}$ and $t \in -i(0, +\infty)$, there exists a unique solution to the RH problem \ref{modelrhp} for $\Psi(z;t)$. Moreover, the $(1,1)$-entry of $\Psi_1(t)$ in \eqref{infinitybehaviorforpsi} is pole-free for $t \in -i(0, +\infty)$.
\end{proposition}

Recalling the Lemma 8.1 obtained in \cite{Dai:Zhai2022}, we have the following representation of the Fredholm determinant in terms of the model RH problem:
\begin{equation}\label{integralforfred}
  \ln \det \left(I - \gamma \mathcal{K}_s^{(\alpha, \beta)}\right) = \int_0^{t} \left(-\frac{1}{2}(\Psi_1(\tau))_{11}-\frac{\alpha^2 - \beta^2}{\tau} \right) d\tau.
\end{equation}

\section{Asymptotic analysis of the model RH problem as $it \to +\infty$}\label{analysistopsi}

From the integral representation of the Fredholm determinant in \eqref{integralforfred}, both large-$t$ and small-$t$ asymptotics of $\Psi_1(t)$ are required to derive the deformed gap probability. In the present work, we intend to derive the large-$t$ asymptotics since the small-$t$ asymptotics is already obtained in \cite{Dai:Zhai2022}. In this section, we perform a Deift-Zhou nonlinear steepest descent method to analyze the model RH problem for $\Psi$ as $it \to +\infty$. The main idea is to covert the original RH problem into a small-norm one via a series of explicit and invertible transformations.

Precisely, we set $it \to + \infty$ and $\gamma \to 1$ such that
\begin{equation}\label{nu}
 \nu = - \ln(1-\gamma) > 0, \qquad \frac{\nu}{|t|} = \frac{1}{2} - (\chi + \alpha) \frac{\ln |t|}{|t|},
\end{equation}
where
\begin{equation}
  \qquad \chi = k+\mu \geq 0, \quad k \in \mathbb{N}_{\geq 0}, \quad \mu \in [-\frac{1}{2},\frac{1}{2}).
\end{equation}

\subsection{Normalization}

To normalize $\Psi(z;t)$ at infinity, we first need to define a $g$-function. Comparing with the one used in \cite[Eq. (4.2)]{Xu:Zhao2020}, we make a modification to reveal the influence of $\gamma$ tending to 1,
\begin{equation}\label{g-function}
  g(z) = \frac{1}{4}\sqrt{z^2-1}-\frac{\chi}{t}\ln D(z), \qquad z \in \mathbb{C} \setminus [-1,1],
\end{equation}
where
\begin{equation}\label{d-function}
  D(z) = \frac{z}{1+i\sqrt{z^2-1}},
\end{equation}
and the principal branch of $\sqrt{z^2-1}$ is taken. After a simple verification, we have the following propositions for $D(z)$.

\begin{proposition}
  \hfill
  \begin{itemize}
    \item [(i)] For $z \in (-1,1)$, we have
    \begin{equation}\label{D+-}
      D_+(z)D_-(z) = 1.
    \end{equation}
    \item [(ii)] As $z \to 0$, we have
    \begin{equation}\label{datzto0}
      D(z) = \left\{
      \begin{aligned}
        & \frac{2}{z}-\frac{z}{2}+O(z^3), &\qquad\im z>0, \\
        & \frac{z}{2}+O(z^3), &\qquad\im z<0. \\
      \end{aligned}
      \right.
    \end{equation}
    \item [(iii)] As $z \to \infty$, we have
    \begin{equation}\label{datztoinfinity}
      D(z) = -i + \frac{1}{z} + O\left(\frac{1}{z^2}\right).
    \end{equation}
  \end{itemize}
\end{proposition}
With the aid of $g(z)$, we introduce the normalization transformation as
\begin{equation}\label{normalization}
  A(z) = e^{\frac{\pi i}{2}\chi \sigma_3} \Psi(z;t) e^{-tg(z)\sigma_3}.
\end{equation}
Then it is readily seen that $A(z)$ satisfies the following RH problem.
\begin{rhp}\label{rhpfora_}
  \quad
\begin{itemize}
  \item [(a)] $A(z)$ is analytic on $z \in \mathbb{C} \setminus \{\cup_{i=1}^7 \Sigma_i \}$, as illustrated in Figure \ref{figure1}.

  \item [(b)] $A(z)$ satisfies the following jump condition
  \begin{equation}
    A_+(z)=A_-(z)J_A(z), \qquad z \in \cup_{i=1}^7 \Sigma_i,
  \end{equation}
  where
  \begin{equation}\label{jumpfora_}
    J_A(z)=
    \begin{cases}
      \begin{pmatrix} 1 & 0 \\ e^{-\pi i(\alpha-\beta)}e^{-2tg(z)} & 1 \end{pmatrix}, &  z \in \Sigma_1,\\
      \begin{pmatrix} 1 & 0 \\ e^{\pi i(\alpha-\beta)}e^{-2tg(z)} & 1 \end{pmatrix}, &  z \in \Sigma_2,\\
      \begin{pmatrix} 1 & -e^{-\pi i(\alpha-\beta)}e^{2tg(z)} \\ 0 & 1 \end{pmatrix}, &  z \in \Sigma_3,\\
      e^{2\pi i \beta \sigma_3}, &  z \in \Sigma_4,\\
      \begin{pmatrix} 1 & -e^{\pi i(\alpha-\beta)}e^{2tg(z)} \\ 0 & 1 \end{pmatrix}, &  z \in \Sigma_5,\\
      \begin{pmatrix} 0 & -e^{\pi i(\alpha-\beta)} \\ e^{-\pi i(\alpha-\beta)} & (1-\gamma)e^{t(g_+(z)-g_-(z))} \end{pmatrix}, &  z \in \Sigma_6,\\
      \begin{pmatrix} 0 & -e^{-\pi i(\alpha-\beta)} \\ e^{\pi i(\alpha-\beta)} & (1-\gamma)e^{t(g_+(z)-g_-(z))} \end{pmatrix}, &  z \in \Sigma_7.\\
    \end{cases}
  \end{equation}

 \item [(c)] As $z \to \infty$, we have
  \begin{equation}\label{infinitybehaviorfora_}
   A(z)=\left(I + \frac{1}{z}\left(e^{\frac{\pi i}{2}\chi \sigma_3} \Psi_1 e^{-\frac{\pi i}{2}\chi \sigma_3}+(\frac{t}{8}+i \chi)\sigma_3\right)
    + O\left(\frac{1}{z^2}\right)\right)z^{-\beta\sigma_3}.
   \end{equation}

 \item [(d)] As $z \to 0, \pm1$, $\left(e^{-\frac{\pi i}{2}\chi\sigma_3}A(z)e^{tg\sigma_3} \right)$ satisfies the same local behaviors as $\Psi(z)$; see \eqref{localbehaviorforpsinear0}, \eqref{localbehaviorforpsinear1} and \eqref{localbehaviorforpsinear-1}.
 \end{itemize}
 \end{rhp}

\subsection{Global parametrix}

Note that as $it \to +\infty$, the jump matrices on $\Sigma_1, \Sigma_2, \Sigma_3, \Sigma_5$ tend to the identity matrix exponentially fast. Then let us take a close look at the $(2,2)$-entries of jump matrices on $\Sigma_6, \Sigma_7$. With \eqref{nu} and \eqref{g-function}, we have for $ z \in (-1,1)$
\begin{equation}
    (1-\gamma)e^{t(g_+(z)-g_-(z))} = \exp \left(-\frac{|t|}{2}(1-\sqrt{1-z^2}) +(\chi+\alpha)\ln |t| - \chi \ln \left(\frac{D_+(z)}{D_-(z)}\right) \right).
\end{equation}
Apparently, when $|z| > \delta$, these entries tend to 0 exponentially fast. 
Therefore, we consider the following RH problem for $P^{(\infty)}(z)$.

\begin{rhp}\label{rhpforpinfinity_}
\quad
 \begin{itemize}
  \item [(a)] $P^{(\infty)}(z)$ is defined and analytic in $\mathbb{C} \setminus \{\Sigma_4
         \cup \Sigma_6 \cup \Sigma_7\}$.
  \item [(b)] $P^{(\infty)}(z)$ satisfies the following jump conditions
        \begin{equation}\label{jumpforpinfinity_}
         P^{(\infty)}_+(z)=P^{(\infty)}_-(z)
          \begin{cases}
           \begin{pmatrix} 0 & -e^{\pi i(\alpha-\beta)} \\ e^{-\pi i(\alpha-\beta)} & 0 \end{pmatrix}, &  z \in \Sigma_6,\\
           \begin{pmatrix} 0 & -e^{-\pi i(\alpha-\beta)} \\ e^{\pi i(\alpha-\beta)} & 0 \end{pmatrix}, &  z \in \Sigma_7.\\
           e^{2\pi i \beta \sigma_3},   & z \in \Sigma_4.  \\
         \end{cases}
        \end{equation}
  \item [(c)] As $z \to \infty$, we have
        \begin{equation}\label{infinitybehaviorforpinfinity_}
         P^{(\infty)}(z)=\left(I + O\left(\frac{1}{z}\right)\right)z^{-\beta\sigma_3},
        \end{equation}
        where the branch of $z^{\beta}$ is taken along the negative imaginary axis such that $\arg z \in (-\frac{\pi}{2}, \frac{3\pi}{2})$.
 \end{itemize}
\end{rhp}

The solution to the above RH problem is given explicitly as
\begin{equation}\label{pinfinity_}
\begin{aligned}
  P^{(\infty)}(z) &= 2^{\beta \sigma_3} e^{\frac{\pi i}{2}(\mu-\beta) \sigma_3}
   \frac{1}{\sqrt{2}}\begin{pmatrix} 1 & -i \\ -i & 1 \end{pmatrix}
  \left(\frac{z-1}{z+1}\right)^{\frac{1}{4}\sigma_3}
  \frac{1}{\sqrt{2}}\begin{pmatrix} 1 & i \\ i & 1 \end{pmatrix}\\
  & \quad\times (z+\sqrt{z^2-1})^{-\beta \sigma_3} e^{\frac{\pi i}{2}\beta \sigma_3}
  \left(\frac{-i+\sqrt{z^2-1}}{z}\right)^{\alpha \sigma_3}D(z)^{\mu \sigma_3},
\end{aligned}
\end{equation}
where the principal branches are taken for $\sqrt{z^2-1}$, $\left(\frac{z-1}{z+1}\right)^{\frac{1}{4}}$, $(\cdot)^{\alpha}$ and $(\cdot)^{\mu}$ such that $\arg (z\pm 1) \in (-\pi, \pi)$ and $\arg z \in (-\pi, \pi)$, while the branch of $z^{\beta}$ is taken along the negative imaginary axis such that $\arg z \in (-\frac{\pi}{2}, \frac{3\pi}{2})$.

\begin{remark}
  Comparing our global parametrix with the one in the undeformed case \cite[Eq. (4.7)]{Xu:Zhao2020}, the only difference is the factor $D(z)^{\mu \sigma_3}$, which is caused by the different choices of $g$-functions. Since we do not impose conditions at the endpoints $0$ and $\pm 1$, the solution to the above RH problem is not unique. It is worth noting that the factor $D(z)^{\mu \sigma_3}$ does not change the jump conditions or the asymptotic behavior of $P^{(\infty)}(z)$ as $z \to \infty$;  see properties of $D(z)$ in \eqref{D+-} and \eqref{datztoinfinity}. This factor is introduced in order to align with the construction of local parametrix near the origin in Section \ref{sec:local-parametrix}.
\end{remark}

The convergence of the global parametrix and $A(z)$ is not uniform near the endpoints $\pm1, 0$, thus we need to further construct the local parametrix near $\pm1, 0$ in the subsequent three subsections. The local parametrices near $\pm1$ are given in terms of the well-known Bessel parametrix, whereas the one near $0$ is given in terms of an orthogonal polynomial parametrix with respect to a weight function which contains a Fisher-Hartwig singularity.

\subsection{Local parametrix near $-1$}

Let $U(z_0, \delta)$ denote the circle centering at $z_0$ with radius $\delta$. First, we intend to look for a local parametrix $P^{(-1)}(z)$ satisfying the following RH problem.

\begin{rhp}\label{rhpforp-1_}
 \quad
 \begin{itemize}
  \item [(a)] $P^{(-1)}(z)$ is defined and analytic in $\overline{U(-1,\delta)} \setminus \{\Sigma_2 \cup \Sigma_3 \cup \Sigma_7 \}$.
  \item [(b)] $P^{(-1)}(z)$ satisfies the following jump condition
        \begin{equation}\label{jumpforp-1}
         P^{(-1)}_+(z)=P^{(-1)}_-(z)J_{A}(z),
         \qquad z \in U(-1,\delta) \cap \{\Sigma_2 \cup \Sigma_3 \cup \Sigma_7 \},
        \end{equation}
        where $J_{A}(z)$ is given in \eqref{jumpfora_}.
  \item [(c)] As $ it \to +\infty$, we have the matching condition
        \begin{equation}\label{matchforp-1_}
         P^{(-1)}(z)= \left(I + O\left(\frac{1}{t}\right)\right)P^{(\infty)}(z), \qquad z \in \partial U(-1, \delta).
        \end{equation}
 \end{itemize}
\end{rhp}

Let us first introduce a conformal mapping near $z=-1$:
\begin{equation}\label{conformalmappingat-1_}
    f^{(-1)}(z) =
      -\left(\frac{1}{4}\sqrt{z^2-1}-\frac{k}{t}(\ln D(z)+\pi i)\right)^2, \qquad z \in U(-1, \delta).
\end{equation}
It is directly seen
\begin{equation}
  f^{(-1)}(z) = \frac{z+1}{8}\left(1-\frac{4k}{|t|}\right)^2\left(1+O\left(z+1\right)^{\frac{1}{2}} \right), \qquad z \to -1.
\end{equation}
This indicates that $f^{(-1)}(z)$ maps the neighborhood near $-1$ on $z$-plane to the neighborhood of 0 on $f^{(-1)}$-plane with no change of direction. Now we are ready to construct the $P^{(-1)}$ with the classical Bessel parametrix.

\begin{lemma}\label{lemmap-1}
  Let $\Phi_B(\zeta)$ be the Bessel parametrix given in Appendix \ref{bessel}. Then, the solution to the RH problem \ref{rhpforp-1_} for $P^{(-1)}$ is given by
  \begin{equation}\label{p-1_}
    P^{(-1)}(z) = E^{(-1)}(z)\Phi_B(|t|^2 f^{(-1)}(z))
    \begin{pmatrix} 1 & \frac{\gamma-1}{2 \pi i}\ln(z+1) \\ 0 & 1 \end{pmatrix} K^{(-1)}(z),
  \end{equation}
  where the branch cut of $\ln (z+1)$ is taken along $(-1, +\infty)$ such that $\arg (z+1) \in (0, 2\pi)$, $\Phi_B$ is defined in \eqref{jb_}, and $K^{(-1)}(z)$ is given as
  \begin{equation}
    K^{(-1)}(z) =
      \begin{cases}
        e^{\frac{\pi i}{2}(\alpha-\beta)\sigma_3}e^{-tg(z)\sigma_3},
        \  & \arg z \in (0, \frac{2}{3}\pi),    \\
        \begin{pmatrix} 1 & 0 \\ -1 & 1 \end{pmatrix}e^{\frac{\pi i}{2}(\alpha-\beta)\sigma_3}e^{-tg(z)\sigma_3},
        \  & \arg z \in (\frac{2}{3}\pi, \pi),    \\
        \begin{pmatrix} 0 & 1 \\ -1 & 1 \end{pmatrix}e^{\frac{\pi i}{2}(\alpha-\beta)\sigma_3}e^{-tg(z)\sigma_3},
        \  & \arg z \in (-\pi, -\frac{2}{3}\pi),    \\
        \begin{pmatrix} 0 & 1 \\ -1 & 0 \end{pmatrix}e^{\frac{\pi i}{2}(\alpha-\beta)\sigma_3}e^{-tg(z)\sigma_3},
        \  & \arg z \in (-\frac{2}{3}\pi, 0).  \\
      \end{cases}
  \end{equation}
 Here $E^{(-1)}(z)$ is an analytic prefactor defined as
  \begin{equation}\label{e-1_}
  \begin{aligned}
    E^{(-1)}(z) &=
     P^{(\infty)}(z)e^{-\frac{\pi i}{2}(\alpha-\beta)\sigma_3}D(z)^{-\mu \sigma_3}e^{k\pi i \sigma_3}\\
     & \quad \times
     \begin{cases}
     \frac{1}{\sqrt{2}}
     \begin{pmatrix} 1 & -i \\ -i & 1 \end{pmatrix}|t|^{\frac{\sigma_3}{2}} (f^{(-1)}(z))^{\frac{1}{4}\sigma_3},  &\im  z > 0, \\
     \frac{1}{\sqrt{2}}
     \begin{pmatrix} i & -1 \\ 1 & -i \end{pmatrix}|t|^{\frac{\sigma_3}{2}}(f^{(-1)}(z))^{\frac{1}{4}\sigma_3},  &\im  z < 0,
    \end{cases}
    \end{aligned}
  \end{equation}
  where the principal branches are taken for $(f^{(-1)}(z))^{\frac{1}{4}}$ and $D(z)^{-\mu}$ such that $\arg (z\pm 1) \in (-\pi, \pi)$ and $\arg z \in (-\pi, \pi)$. 
\end{lemma}

\begin{proof}
  We first demonstrate the analyticity of $E^{(-1)}(z)$. According to the definition in \eqref{e-1_}, the possible jump contour is $(-1-\delta, -1+\delta)$. For $z \in (-1, -1+\delta)$, it follows from \eqref{d-function}, \eqref{pinfinity_} and \eqref{conformalmappingat-1_} that
  \begin{equation}
  \begin{aligned}
       \left(E_-^{(1)}(z)\right)^{-1}E_+^{(1)}(z)
     &= \frac{1}{2}(f^{(-1)}(z))^{-\frac{1}{4}\sigma_3}\begin{pmatrix} -i & 1 \\ -1 & i \end{pmatrix}
     \begin{pmatrix} 0 & -1 \\ 1 & 0 \end{pmatrix}\begin{pmatrix} 1 & -i \\ -i & 1 \end{pmatrix}(f^{(-1)}(z))^{\frac{1}{4}\sigma_3}\\
      &= I.
     \end{aligned}
  \end{equation}
  For $z \in (-1-\delta,-1)$, the verification is similar. When $z=-1$, by \eqref{datztoinfinity}, \eqref{pinfinity_}, \eqref{e-1_}, it is easily seen that $E^{(-1)}(-1)=O(1)$. Therefore, $z=-1$ is a removable singularity and $E^{(-1)}(z)$ is analytic in $U(-1,\delta)$.

  The jump conditions for $P^{(-1)}(z)$ are straightforward to verify by applying \eqref{g-function}, \eqref{conformalmappingat-1_} and \eqref{p-1_}. For $z \in \Sigma_7$, we have
  \begin{equation}
   \begin{aligned}
     & \left(P_-^{(-1)}(z)\right)^{-1}P_+^{(-1)}(z)  \\
     & = e^{tg_-(z)\sigma_3}e^{-\frac{\pi i}{2}(\alpha-\beta)\sigma_3}
     \begin{pmatrix} 0 & -1 \\ 1 & 0 \end{pmatrix}
     \begin{pmatrix} 1 & \frac{\gamma-1}{2 \pi i}(\ln(z+1)_+ - \ln(z+1)_-) \\ 0 & 1 \end{pmatrix}
     e^{\frac{\pi i}{2}(\alpha-\beta)\sigma_3}\\
     &\quad \times e^{-tg_+(z)\sigma_3}   \\
     & = \begin{pmatrix} 0 & -e^{-\pi i(\alpha-\beta)} \\ e^{\pi i(\alpha-\beta)} & (1-\gamma)e^{t(g_+(z)-g_-(z))} \end{pmatrix},
   \end{aligned}
  \end{equation}
  which matches with the jump of $A(z)$ on $\Sigma_7$. For $z$ on the other jump contours, the jump condition can be verified by the similar procedure.

  Finally, let us check the matching condition. As $it \to +\infty$, considering $\arg (z+1) \in (0,\frac{2}{3}\pi)$, it follows from \eqref{e-1_} and \eqref{besselinfinity} that
  \begin{equation}
   \begin{aligned}
     & P^{(-1)}(z)\left( P^{(\infty)}(z) \right)^{-1}       \\
     & = P^{(\infty)}(z)e^{-\frac{\pi i}{2}(\alpha-\beta)\sigma_3}D(z)^{-\mu\sigma_3}e^{k\pi i \sigma_3}\left(I + O\left(\frac{1}{t}\right)\right) \\
     & \qquad \times
     \begin{pmatrix} 1 & \frac{\gamma-1}{2\pi i} e^{2|t|\sqrt{f^{(-1)}(z)}} \ln (z+1) \\ 0 & 1 \end{pmatrix}
     D(z)^{\mu\sigma_3}e^{-k\pi i \sigma_3}e^{\frac{\pi i}{2}(\alpha-\beta)\sigma_3}\left(P^{(\infty)}(z)\right)^{-1}\\
     & = I + O\left(\frac{1}{t}\right).
   \end{aligned}
  \end{equation}
Here we apply the facts that $\gamma$ tends to 1 exponentially fast and $P^{(\infty)}(z)D(z)^{-\mu\sigma_3}$ is independent of $t$. For $z$ in other sectors, the matching condition can be verified following the same procedure. This completes our proof.
\end{proof}

\subsection{Local parametrix near 1}
Similar to the scenario near $z=-1$, we intend to find a local parametrix in $U(1,\delta)$ solving the following RH problem for $P^{(1)}(z)$.
\begin{rhp}\label{rhpforp1_}
 \quad
 \begin{itemize}
  \item [(a)] $P^{(1)}(z)$ is defined and analytic in $\overline{U(1,\delta)} \setminus \{\Sigma_1 \cup \Sigma_5 \cup \Sigma_6 \}$.

  \item [(b)] $P^{(1)}(z)$ satisfies the following jump condition
        \begin{equation}\label{jumpforp1_}
         P^{(1)}_+(z)=P^{(1)}_-(z)J_{A}(z),
         \qquad z \in U(1,\delta) \cap \{\Sigma_1 \cup \Sigma_5 \cup \Sigma_6 \}.
        \end{equation}
        where $J_{A}(z)$ is given in \eqref{jumpfora_}.

  \item [(c)] As $it \to +\infty$, we have the matching condition
        \begin{equation}\label{matchforp1_}
         P^{(1)}(z)= \left(I + O\left(\frac{1}{t}\right)\right)P^{(\infty)}(z), \qquad z \in \partial U(1, \delta).
        \end{equation}
 \end{itemize}
\end{rhp}

This parametrix construction is similar to that in $U(-1, \delta)$, which is also given in terms of the Bessel functions. We introduce a conformal mapping near $z=1$:
\begin{equation}\label{conformalmappingat1_}
    f^{(1)}(z) = -\left(\frac{1}{4}\sqrt{z^2-1}-\frac{k}{t}\ln D(z)\right)^2, \qquad z \in U(-1, \delta).
\end{equation}
Then we have
\begin{equation}
  f^{(1)}(z) = -\frac{z-1}{8}\left(1-\frac{4k}{|t|}\right)^2\left(1+O\left(z-1\right)^{\frac{1}{2}} \right), \qquad z \to 1.
\end{equation}
Apparently, $f^{(1)}(z)$ maps the neighborhood near 1 on $z$-plane to the neighborhood near 0 on $f^{(1)}$-plane with a rotation of $\pi$. Then, the solution to the  above RH problem is given in the following lemma.

\begin{lemma}
  Let $\Phi_B(\zeta)$ be the Bessel parametrix given in Appendix \ref{bessel}. Then, the solution to the above RH problem for $P^{(1)}(z)$ is given by
  \begin{equation}\label{p1_}
    P^{(1)}(z) = E^{(1)}(z)\sigma_3 \Phi_B(|t|^2 f^{(1)}(z))
    \begin{pmatrix} 1 & \frac{\gamma-1}{2 \pi i}\ln(z-1) \\ 0 & 1 \end{pmatrix}\sigma_3  K^{(1)}(z),
  \end{equation}
 where the principal branch of $\ln (z-1)$ is taken, $\Phi_B$ is given in \eqref{jb_}, and $K^{(1)}(z)$ is given as
  \begin{equation}
    K^{(1)}(z) =
      \begin{cases}
        \begin{pmatrix} 1 & 0 \\ -1 & 1 \end{pmatrix}e^{-\frac{\pi i}{2}(\alpha-\beta)\sigma_3}e^{-tg(z)\sigma_3},
        \quad & \arg z \in (0, \frac{\pi}{3}),    \\
        e^{-\frac{\pi i}{2}(\alpha-\beta)\sigma_3}e^{-tg(z)\sigma_3},
        \quad & \arg z \in (\frac{\pi}{3}, \pi),    \\
        \begin{pmatrix} 0 & 1 \\ -1 & 0 \end{pmatrix}e^{-\frac{\pi i}{2}(\alpha-\beta)\sigma_3}e^{-tg(z)\sigma_3},
        \quad & \arg z \in (-\pi, -\frac{\pi}{3}),    \\
        \begin{pmatrix} 0 & 1 \\ -1 & 1 \end{pmatrix}e^{-\frac{\pi i}{2}(\alpha-\beta)\sigma_3}e^{-tg(z)\sigma_3},
        \quad & \arg z \in (-\frac{\pi}{3}, 0).  \\
      \end{cases}
  \end{equation}
 Moreover, $E^{(1)}(z)$ is an analytic prefactor defined as
  \begin{equation}\label{e1_}
    E^{(1)}(z) = P^{(\infty)}(z)e^{\frac{\pi i}{2}(\alpha-\beta)\sigma_3}D(z)^{-\mu\sigma_3}
    \begin{cases}
     \frac{1}{\sqrt{2}}\begin{pmatrix} 1 & i \\ i & 1 \end{pmatrix}
      |t|^{\frac{\sigma_3}{2}}(f^{(1)}(z))^{\frac{1}{4}\sigma_3},  &\im  z > 0, \\
     \frac{1}{\sqrt{2}}\begin{pmatrix} -i & -1 \\ 1 & i \end{pmatrix}
      |t|^{\frac{\sigma_3}{2}}(f^{(1)}(z))^{\frac{1}{4}\sigma_3},  &\im  z < 0,
    \end{cases}
  \end{equation}
where $(f^{(1)}(z))^{\frac{1}{4}}$ and $D(z)^{-\mu}$ take the principal branches.
\end{lemma}

\begin{proof}
 Let us verify the analyticity of $E^{(1)}(z)$ in $U(1,\delta)$. According to the definition in \eqref{e1_}, the possible jump contour is $(1-\delta, 1+\delta)$. For $z \in (1, 1+\delta)$, \eqref{conformalmappingat1_} and \eqref{e1_} yield
  \begin{equation}
     \left(E_-^{(1)}(z)\right)^{-1}E_+^{(1)}(z)
   = (f^{(1)}(z))_-^{-\frac{1}{4}\sigma_3}
   \begin{pmatrix} i & 0 \\ 0 & -i  \end{pmatrix}(f^{(1)}(z))_+^{\frac{1}{4}\sigma_3} = I.
 \end{equation}
For $z \in (1-\delta, 1)$, the verification is similar. Moreover, it is easily seen that $E^{(1)}(1)=O(1)$ from \eqref{datztoinfinity}, \eqref{pinfinity_} and \eqref{e1_} . Therefore, $z=1$ is a removable singularity and $E^{(1)}(z)$ is indeed analytic. The rest of the proof is similar to that of Lemma \ref{lemmap-1}, we omit the details.

  This completes our proof.
\end{proof}

\subsection{Local parametrix at 0} \label{sec:local-parametrix}
It remains to construct a local parametrix near the origin. This local parametrix construction is vital in our analysis. We intend to find a function $P^{(0)}(z)$ satisfying the following RH problem.

\begin{rhp}\label{rhpforp0_}
\quad
 \begin{itemize}
  \item [(a)] $P^{(0)}(z)$ is defined and analytic in $\overline{U(0,\delta)} \setminus \{\Sigma_4 \cup \Sigma_6 \cup \Sigma_7 \}$.

  \item [(b)] $P^{(0)}(z)$ satisfies the jump condition
        \begin{equation}\label{jumpforp0_}
         P^{(0)}_+(z)=P^{(0)}_-(z)J_{A}(z),
         \qquad z \in U(0,\delta) \cap \{\Sigma_4 \cup \Sigma_6 \cup \Sigma_7 \},
        \end{equation}
        where $J_{A}(z)$ is given in \eqref{jumpfora_}.

  \item [(c)] As $ it \to +\infty$, we have the matching condition
        \begin{equation}\label{matchforp0_}
         P^{(0)}(z)= \left(I + O\left(t^{-\frac{1}{2}+|\mu|}\right)\right)P^{(\infty)}(z), \qquad z \in \partial U(0,\delta).
        \end{equation}
 \end{itemize}
\end{rhp}

This local parametrix is constructed in terms of orthogonal polynomials; also see \cite{Both2016}. We first define the $2 \times 2$ orthogonal polynomial parametrix as
\begin{equation}\label{hermiteparametrix}
  H(\zeta)=
  \begin{pmatrix}
    \pi_k(\zeta) & \frac{1}{2 \pi i}\int_{\mathbb{R}}\frac{\pi_k(z)w(z)}{z-\zeta}dz \\
    \gamma_{k-1}\pi_{k-1}(\zeta) & \frac{\gamma_{k-1}}{2 \pi i}\int_{\mathbb{R}}\frac{\pi_{k-1}(z)w(z)}{z-\zeta}dz
  \end{pmatrix}e^{-\frac{1}{2}\zeta^2 \sigma_3}, \qquad \zeta \in \mathbb{C}\setminus \mathbb{R}, \quad k \in \mathbb{Z}_{\geq 0},
\end{equation}
where $\pi_{-1}(\zeta) \equiv 0$, and $\pi_k(\zeta)$ are the monic orthogonal polynomials given in \eqref{eq:hn-def}.

With the properties of orthogonal polynomials, it is easily seen that $h_n^{-\frac{1}{2}}$ is the leading coefficient of the $n$-th orthonormal polynomial. It is well-known that the above function satisfies a RH problem as follows.
\begin{rhp}\label{rhpforh}
  The $2 \times 2$ matrix-valued function defined in \eqref{hermiteparametrix} satisfies the following properties:
  \begin{itemize}
    \item [(a)] $H(\zeta)$ is analytic for $\zeta \in \mathbb{C}\setminus \mathbb{R}$.

    \item [(b)] $H(\zeta)$ satisfies the jump condition
    \begin{equation}\label{jumpforh}
      H_+(\zeta) = H_-(\zeta)\begin{pmatrix}  1 & w(\zeta) \, e^{\zeta^2}\\ 0 & 1 \end{pmatrix}, \quad \zeta \in \mathbb{R},
    \end{equation}
    where $w(\zeta)$ is given in \eqref{hermiteweight}.

    \item [(c)] As $\zeta \to \infty$, we have
    \begin{equation}\label{hermitinfinity}
      H(\zeta) = \left(
      I + \frac{H_1}{\zeta} + \frac{H_2}{\zeta^2} + O\left(\frac{1}{\zeta^3}\right)
      \right)\zeta^{k \sigma_3}e^{-\frac{1}{2}\zeta^2 \sigma_3}.
    \end{equation}
    Here, $H_1, H_2$ are the coefficient matrices. According to the properties of orthogonal polynomials, $H_1$ admits the form as
    \begin{equation}\label{h1}
      H_1 = \begin{pmatrix}  a_k & \gamma_k^{-1} \\ \gamma_{k-1} & d_k \end{pmatrix},
    \end{equation}
    where $\gamma_k$ are given by \eqref{defgammak}, while $a_k,d_k$ are insignificant constants.
  \end{itemize}
\end{rhp}

Then, we introduce a local conformal mapping in $U(0,\delta)$:
\begin{equation}\label{conformalmappingat0_}
  f^{(0)}(z) =
   \begin{cases}
    \left(\frac{i}{2}\sqrt{z^2-1}+\frac{1}{2}\right)^{\frac{1}{2}}, & \qquad \im{z} > 0, \\
    \left(-\frac{i}{2}\sqrt{z^2-1}+\frac{1}{2}\right)^{\frac{1}{2}}, & \qquad \im{z} < 0.
   \end{cases}
\end{equation}
It is easily seen that
\begin{equation}
  f^{(0)}(z) = \frac{z}{2}(1+O(z)), \qquad z \to 0.
\end{equation}
Apparently,  $f^{(0)}(z)$ maps the neighborhood near the origin on $z$-plane to the neighborhood near the origin on $f^{(0)}$-plane, preserving the argument.

With the orthogonal polynomial parametrix and the conformal mapping, we are able to construct the solution to the RH problem \eqref{rhpforp0_}.
\begin{lemma}
  Let $H(\zeta)$ be the parametrix given in \eqref{hermiteparametrix}. Then, the solution to the RH problem \ref{rhpforp0_} for $P^{(0)}(z)$ is given as
  \begin{equation}\label{p0_}
    P^{(0)}(z)= E^{(0)}(z)|t|^{\frac{\mu}{2}\sigma_3}H(|t|^{\frac{1}{2}}f^{(0)}(z))(1-\gamma)^{-\frac{1}{2}\sigma_3}(f^{(0)}(z))^{\alpha \sigma_3}|t|^{\frac{\alpha}{2}\sigma_3}K^{(0)}(z),
  \end{equation}
   where the branch of $(\cdot)^{\alpha}$ is taken along $(0,+\infty)$, and $K^{(0)}(z)$ is given as
  \begin{equation}
   K^{(0)}(z)=
   \begin{cases}
    e^{-\frac{\pi i}{2}(\alpha+\beta)\sigma_3}e^{-tg(z)\sigma_3}, \ & \arg z \in (0,\pi),      \\
    e^{-\frac{\pi i}{2}(3\alpha-\beta)\sigma_3}\begin{pmatrix} 0 & 1 \\ -1 & 0 \end{pmatrix}e^{-tg(z)\sigma_3},    \ & \arg z \in (-\pi, -\frac{\pi}{2}),      \\
    e^{-\frac{3\pi i}{2}(\alpha+\beta)\sigma_3}\begin{pmatrix} 0 & 1 \\ -1 & 0 \end{pmatrix}e^{-tg(z)\sigma_3},    \ & \arg z \in (-\frac{\pi}{2},0).
   \end{cases}
  \end{equation}
  Moreover, the analytic prefactor $E^{(0)}(z)$ is defined as
  \begin{equation}\label{e0_}
  \begin{aligned}
    E^{(0)}(z)&=P^{(\infty)}(z)\\
    &\quad \times
    \begin{cases}
      (f^{(0)}(z))^{-(k+\alpha)\sigma_3} D(z)^{-\chi \sigma_3} e^{\frac{\pi i}{2}(\alpha+\beta)\sigma_3},    \   & \arg z \in (0,\pi),  \\
      \begin{pmatrix} 0 & -1 \\ 1 & 0 \end{pmatrix}
      (f^{(0)}(z))^{-(k+\alpha)\sigma_3} D(z)^{\chi \sigma_3} e^{\frac{\pi i}{2}(3\alpha-\beta)\sigma_3},    \   & \arg z \in  (-\pi, -\frac{\pi}{2}),  \\
      \begin{pmatrix} 0 & -1 \\ 1 & 0 \end{pmatrix}
      (f^{(0)}(z))^{-(k+\alpha)\sigma_3} D(z)^{\chi \sigma_3} e^{\frac{3\pi i}{2}(\alpha+\beta)\sigma_3},     \   & \arg z \in (-\frac{\pi}{2},0), \\
    \end{cases}
    \end{aligned}
  \end{equation}
  where $(\cdot)^{\chi}$ takes the principal branch.
\end{lemma}

\begin{proof}
  Again, we follow the same process to organize the proof. First we verify the analyticity of $E^{(0)}(z)$. From the definition in \eqref{e0_}, it is obvious that $E^{(0)}(z)$ is analytic in $U(0,\delta) \setminus \{\Sigma_4 \cup \Sigma_6 \cup \Sigma_7 \}$. For $z \in (0, \delta)$, recalling \eqref{jumpforpinfinity_} and \eqref{conformalmappingat0_}, we have
  \begin{equation}
    \left(E_-^{(0)}(z)\right)^{-1}E_+^{(0)}(z)
    = e^{-2\pi i \alpha\sigma_3} (f^{(0)}_-(z))^{(k+\alpha)\sigma_3} (f^{(0)}_+(z))^{-(k+\alpha)\sigma_3} = I.
  \end{equation}
  For the other potential jump contours, the verification is similar. When $z=0$, by \eqref{datzto0}, \eqref{pinfinity_} and \eqref{e0_}, it is easily seen that $E^{(0)}(0)=O(1)$. Therefore, $z=0$ is a removable singularity, and $E^{(0)}(z)$ is indeed analytic in $U(0,\delta)$.

  It is straightforward to verify that $P^{(0)}(z)$ satisfies the jump conditions with the analyticity of the prefactor $E^{(0)}(z)$ in $U(0,\delta)$. For $z \in [0,\delta)$, it follows from \eqref{jumpforh} and \eqref{conformalmappingat0_} that
  \begin{equation}
   \begin{aligned}
    & \left(P^{(0)}_-(z)\right)^{-1}P^{(0)}_+(z)  \\
    & = e^{tg_-(z)\sigma_3}\begin{pmatrix} 0 & -1 \\ 1 & 0 \end{pmatrix}e^{\frac{\pi i}{2}(\alpha+3\beta)\sigma_3}(1-\gamma)^{\frac{1}{2}\sigma_3}
    \begin{pmatrix}  e^{-2\pi i \alpha} & e^{-2\pi i \beta} \\ 0 & e^{2\pi i \alpha} \end{pmatrix}
    (1-\gamma)^{-\frac{1}{2}\sigma_3}\\
    & \quad \times e^{\frac{\pi i}{2}(\alpha-\beta)\sigma_3}e^{-tg_+(z)\sigma_3}    \\
    & = \begin{pmatrix} 0 & -e^{\pi i(\alpha-\beta)} \\ e^{-\pi i(\alpha-\beta)} & (1-\gamma)e^{t(g_+(z)-g_-(z))} \end{pmatrix}.
  \end{aligned}
 \end{equation}
 For $z$ on the other contours, the jump condition can be verified by the same process.

 It remains to verify the matching condition \eqref{matchforp0_}. With the analyticity of $E^{(0)}(z)$ and the fact that $f^{(0)}(z)$ is $t$-independent, a combination of \eqref{pinfinity_}, \eqref{hermitinfinity}, \eqref{p0_} and \eqref{e0_} gives
 \begin{small}
 \begin{equation}\label{p0pinfinity}
 \begin{aligned}
      & P^{(0)}(z)\left( P^{(\infty)}(z) \right)^{-1}\\
   & =  E^{(0)}(z)t^{\frac{|\mu|}{2}\sigma_3}\left(I + \frac{1}{|t|^{\frac{1}{2}}f^{(0)}(z)}
    \begin{pmatrix}  * & \gamma_k^{-1} \\ \gamma_{k-1} & * \end{pmatrix} + O\left(\frac{1}{t}\right)\right)t^{-\frac{|\mu|}{2}\sigma_3}
    (E^{(0)}(z))^{-1}.
    \end{aligned}
 \end{equation}
 \end{small}
According to \eqref{e0_}, it is easily seen that $E^{(0)}(z)$ is $t$-independent. Hence, the matching condition is verified. This completes our proof.
\end{proof}


\subsection{Ratio RH problem}

With all of the local parametrices constructed, the ratio transformation is defined as
\begin{equation}\label{r!}
 R(z)=A(z)
 \begin{cases}
   \left(P^{(0)}(z)\right)^{-1}, \qquad      & z \in U(0,\delta), \\
   \left(P^{(1)}(z)\right)^{-1}, \qquad      & z \in U(1,\delta),  \\
   \left(P^{(-1)}(z)\right)^{-1}, \qquad     & z \in U(-1,\delta), \\
   \left(P^{(\infty)}(z)\right)^{-1}, \qquad & \mathrm{otherwise}. \\
 \end{cases}
\end{equation}
It is straightforward to see $R(z)$ satisfies the following RH problem.

\begin{rhp}
  \quad
\begin{itemize}
  \item [(a)] $R(z)$ is analytic in $\mathbb{C} \setminus \Sigma_{R}$, where $\Sigma_R$ is illustrated in Figure \ref{figure2}.

 \begin{figure}[t]
 \center
\tikzset{every picture/.style={line width=0.75pt}} 

\begin{tikzpicture}[x=0.75pt,y=0.75pt,yscale=-0.7,xscale=0.7]

\draw   (100,147) .. controls (100,127.12) and (116.12,111) .. (136,111) .. controls (155.88,111) and (172,127.12) .. (172,147) .. controls (172,166.88) and (155.88,183) .. (136,183) .. controls (116.12,183) and (100,166.88) .. (100,147) -- cycle ;
\draw   (240,148) .. controls (240,128.12) and (256.12,112) .. (276,112) .. controls (295.88,112) and (312,128.12) .. (312,148) .. controls (312,167.88) and (295.88,184) .. (276,184) .. controls (256.12,184) and (240,167.88) .. (240,148) -- cycle ;
\draw   (381,147) .. controls (381,127.12) and (397.12,111) .. (417,111) .. controls (436.88,111) and (453,127.12) .. (453,147) .. controls (453,166.88) and (436.88,183) .. (417,183) .. controls (397.12,183) and (381,166.88) .. (381,147) -- cycle ;
\draw    (132,111) -- (137,111) ;
\draw [shift={(139.5,111)}, rotate = 180] [fill={rgb, 255:red, 0; green, 0; blue, 0 }  ][line width=0.08]  [draw opacity=0] (8.93,-4.29) -- (0,0) -- (8.93,4.29) -- cycle    ;
\draw    (274,112) -- (279,112) ;
\draw [shift={(281.5,112)}, rotate = 180] [fill={rgb, 255:red, 0; green, 0; blue, 0 }  ][line width=0.08]  [draw opacity=0] (8.93,-4.29) -- (0,0) -- (8.93,4.29) -- cycle    ;
\draw    (412,111) -- (417,111) ;
\draw [shift={(419.5,111)}, rotate = 180] [fill={rgb, 255:red, 0; green, 0; blue, 0 }  ][line width=0.08]  [draw opacity=0] (8.93,-4.29) -- (0,0) -- (8.93,4.29) -- cycle    ;
\draw    (34,50) -- (112,121) ;
\draw [shift={(76.7,88.87)}, rotate = 222.31] [fill={rgb, 255:red, 0; green, 0; blue, 0 }  ][line width=0.08]  [draw opacity=0] (8.93,-4.29) -- (0,0) -- (8.93,4.29) -- cycle    ;
\draw    (440,175) -- (518,246) ;
\draw [shift={(482.7,213.87)}, rotate = 222.31] [fill={rgb, 255:red, 0; green, 0; blue, 0 }  ][line width=0.08]  [draw opacity=0] (8.93,-4.29) -- (0,0) -- (8.93,4.29) -- cycle    ;
\draw    (35,240) -- (114,175) ;
\draw [shift={(78.36,204.32)}, rotate = 140.55] [fill={rgb, 255:red, 0; green, 0; blue, 0 }  ][line width=0.08]  [draw opacity=0] (8.93,-4.29) -- (0,0) -- (8.93,4.29) -- cycle    ;
\draw    (441,120) -- (520,55) ;
\draw [shift={(484.36,84.32)}, rotate = 140.55] [fill={rgb, 255:red, 0; green, 0; blue, 0 }  ][line width=0.08]  [draw opacity=0] (8.93,-4.29) -- (0,0) -- (8.93,4.29) -- cycle    ;
\draw  [fill={rgb, 255:red, 0; green, 0; blue, 0 }  ,fill opacity=1 ] (134.51,147.74) .. controls (134.51,147.33) and (134.85,147) .. (135.26,147) .. controls (135.67,147) and (136,147.33) .. (136,147.74) .. controls (136,148.15) and (135.67,148.49) .. (135.26,148.49) .. controls (134.85,148.49) and (134.51,148.15) .. (134.51,147.74) -- cycle ;
\draw  [fill={rgb, 255:red, 0; green, 0; blue, 0 }  ,fill opacity=1 ] (276,148) .. controls (276,147.59) and (276.33,147.26) .. (276.74,147.26) .. controls (277.15,147.26) and (277.49,147.59) .. (277.49,148) .. controls (277.49,148.41) and (277.15,148.74) .. (276.74,148.74) .. controls (276.33,148.74) and (276,148.41) .. (276,148) -- cycle ;
\draw  [fill={rgb, 255:red, 0; green, 0; blue, 0 }  ,fill opacity=1 ] (416.26,147.74) .. controls (416.26,147.33) and (416.59,147) .. (417,147) .. controls (417.41,147) and (417.74,147.33) .. (417.74,147.74) .. controls (417.74,148.15) and (417.41,148.49) .. (417,148.49) .. controls (416.59,148.49) and (416.26,148.15) .. (416.26,147.74) -- cycle ;

\draw (273,150) node [anchor=north west][inner sep=0.75pt]   [align=left] {0};
\draw (414,150) node [anchor=north west][inner sep=0.75pt]   [align=left] {1};
\draw (127,150) node [anchor=north west][inner sep=0.75pt]   [align=left] {$-1$};
\draw (458,65) node [anchor=north west][inner sep=0.75pt]   [align=left] {$\Sigma_1$};
\draw (77,65) node [anchor=north west][inner sep=0.75pt]   [align=left] {$\Sigma_2$};
\draw (77,212) node [anchor=north west][inner sep=0.75pt]   [align=left] {$\Sigma_3$};
\draw (461,212) node [anchor=north west][inner sep=0.75pt]   [align=left] {$\Sigma_5$};

\end{tikzpicture}
\caption{Contours for $R(z)$.}
    \label{figure2}
  \end{figure}
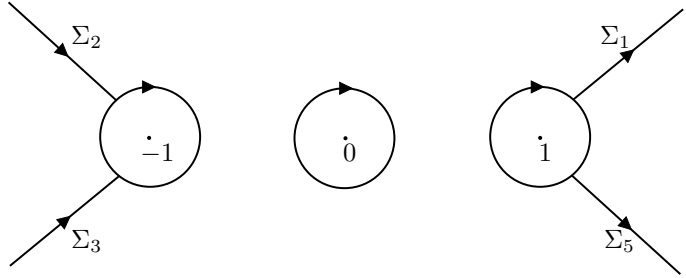

  \item [(b)] $R(z)$ satisfies the jump condition
        \begin{equation}\label{jumpforr!}
          R_+(z)=R_-(z)J_{R}(z), \qquad z \in \Sigma_R,
        \end{equation}
        where
        \begin{equation}\label{jr!}
          J_{R}(z)=
          \begin{cases}
            P^{(1)}(z)\left(P^{(\infty)}(z)\right)^{-1}, \qquad & z \in \partial U(1,\delta),    \\
            P^{(-1)}(z)\left(P^{(\infty)}(z)\right)^{-1}, \qquad & z \in \partial U(-1,\delta),    \\
            P^{(0)}(z)\left(P^{(\infty)}(z)\right)^{-1}, \qquad & z \in \partial U(0,\delta),    \\
            P^{(\infty)}(z)J_A(z)\left(P^{(\infty)}(z)\right)^{-1}, \qquad & \mathrm{otherwise}.
          \end{cases}
        \end{equation}
        Here the jump contours $\partial U(j, \delta)$, $j = \pm1,0$ are clockwise-oriented.
\item [(c)] As $z \to \infty$, we have
      \begin{equation}
        R(z) = I + O(z^{-1}).
      \end{equation}
\end{itemize}
\end{rhp}

It is easily seen from \eqref{jumpfora_} and \eqref{pinfinity_} that $P^{(\infty)}(z)J_A(z)\left(P^{(\infty)}(z)\right)^{-1}$ tends to the identity matrix exponentially as $it \to +\infty$, uniformly for $z \in \Sigma_R \setminus \{\partial U(-1,\delta) \cup \partial U(0,\delta) \cup \partial U(1,\delta) \}$. Hence, it follows from  \eqref{matchforp-1_}, \eqref{matchforp1_} and \eqref{matchforp0_} that as $it \to +\infty$
\begin{equation}
  J_R(z) = I + \left(\frac{1}{t}\right), \qquad z \in \Sigma_R \setminus \partial U(0,\delta).
\end{equation}
When $z \in \partial U(0,\delta)$ a detailed form of $J_{R}(z)$ as $it \to +\infty$ is obtained via \eqref{p0pinfinity}:
\begin{equation}\label{jrasztoinfinity_}
    J_{R}(z) = I + \frac{|t|^{-\frac{1}{2}}}{f^{(0)}(z)}E^{(0)}(z)
    \begin{pmatrix}
      a_k & \gamma_k^{-1} |t|^{\mu} \\ \gamma_{k-1}|t|^{-\mu} & d_k
    \end{pmatrix}
    \left(E^{(0)}(z)\right)^{-1} + O(t^{-1+|\mu|}).
\end{equation}
Clearly, for $\mu \in [-\frac{1}{2},\frac{1}{2})$, the off-diagonal entries of $J_{R}(z)-I$ do not tend to 0 uniformly.

In order to resolve with this issue, we define the following $t$-independent matrices.
\begin{align}
  R_1^+ & = \frac{\gamma_k^{-1}}{f^{(0)}(z)}E^{(0)}(z)
  \begin{pmatrix} 0 & 1 \\ 0 & 0 \end{pmatrix}\left(E^{(0)}(z)\right)^{-1},\label{r1+}\\
  R_1^- & = \frac{\gamma_{k-1}}{f^{(0)}(z)}E^{(0)}(z)
  \begin{pmatrix} 0 & 0 \\ 1 & 0 \end{pmatrix}\left(E^{(0)}(z)\right)^{-1},\label{r1-}\\
  R_2 & = \frac{1}{f^{(0)}(z)}E^{(0)}(z)\begin{pmatrix} a_k & 0 \\ 0 & d_k \end{pmatrix}\left(E^{(0)}(z)\right)^{-1}.\label{r2}
\end{align}
Substituting the above matrices into \eqref{jrasztoinfinity_}, we have as $it \to +\infty$
\begin{equation}
  J_{R}(z)-I=R_1^+|t|^{-\frac{1}{2}+\mu} + R_1^-|t|^{-\frac{1}{2}-\mu} + R_2|t|^{-\frac{1}{2}} + O(t^{-1+|\mu|}).
\end{equation}
It is now clear to see that the first two terms of $J_{R}(z)-I$ do not tend to 0 uniformly as $it \to +\infty$. To eliminate these factors, we define
\begin{align}
  E^+(z) & = I+R_1^+|t|^{-\frac{1}{2}+\mu}, \qquad \mu \in [0,\frac{1}{2}),\label{e+}    \\
  E^-(z) & = I+R_1^-|t|^{-\frac{1}{2}-\mu}, \qquad \mu \in [-\frac{1}{2},0).\label{e-}
\end{align}
Directly, it follows that
\begin{align}
  \left(E^+(z)\right)^{-1} & = I-R_1^+|t|^{-\frac{1}{2}+\mu},
  \qquad \mu \in [0,\frac{1}{2}),\label{e+inverse}   \\
  \left(E^-(z)\right)^{-1} & = I-R_1^-|t|^{-\frac{1}{2}-\mu}, \qquad \mu \in [-\frac{1}{2},0).\label{e-inverse}
\end{align}

\subsection{Singular RH problem}

In this step, we define
\begin{equation}\label{q}
  Q(z) = R(z) \begin{cases}
    E^+(z), \qquad & |z|<\delta, \  \mu \in [0,\frac{1}{2}), \\
    E^-(z), \qquad & |z|<\delta, \  \mu \in [-\frac{1}{2},0), \\
    I, & \mathrm{otherwise}.
  \end{cases}
\end{equation}
Then, $Q(z)$ leads us to a singular RH problem.

\begin{rhp}\label{rhpforq}
\quad
  \begin{itemize}
    \item [(a)] $Q(z)$ is analytic in $\mathbb{C} \setminus \left\{ \Sigma_{R} \cup \{0\}\right\}$, where $\Sigma_{R}$ denotes the jump contours of $R(z)$.
    \item [(b)] $Q(z)$ satisfies the jump conditions
    \begin{equation}
      Q_+(z) = Q_-(z) \begin{cases}
        \left(E^+(z)\right)^{-1}J_{R}(z), \qquad & |z| = \delta, \ \mu \in [0,\frac{1}{2}), \\
        \left(E^-(z)\right)^{-1}J_{R}(z), \qquad & |z| = \delta, \ \mu \in [-\frac{1}{2},0), \\
        J_{R}(z),  & \mathrm{otherwise}.
      \end{cases}
    \end{equation}

    \item [(c)] $Q(z)$ has a first order pole at $z = 0$. More precisely, when $z \to 0$, we have
    \begin{equation}\label{qat0}
      Q(z) = \hat{Q}(z)
      \begin{cases}
        \left(I + \frac{\sigma_+}{z}
        \begin{pmatrix} -1 & 1 \\ -1 & 1 \end{pmatrix} + O(z)\right)T^{-1}, &\quad \mu \in [0,\frac{1}{2}), \\
        \left(I + \frac{\sigma_-}{z}
        \begin{pmatrix} -1 & -1 \\ 1 & 1 \end{pmatrix} + O(z)\right)T^{-1}, &\quad \mu \in [-\frac{1}{2},0),
      \end{cases}
    \end{equation}
    where $\hat{Q}(z)$ is analytic at $z = 0$, $\sigma_\pm$ are constants independent of $z$, and $T$ is a constant matrix defined as
    \begin{equation}
        T = 2^{\beta \sigma_3}e^{\frac{\pi i}{2}(\alpha - \beta)\sigma_3}.
    \end{equation}

    \item [(d)] As $z \to \infty$, we have
    \begin{equation}
      Q(z) = I + O(z^{-1}).
    \end{equation}
  \end{itemize}
\end{rhp}

Directly, with the fact that $R(z)$ is actually analytic at $z = 0$, we have the explicit expression for $\sigma_{\pm}$ in the following proposition.

\begin{proposition}
  The constants $\sigma_\pm$ in \eqref{qat0} are given as
  \begin{align}
      \sigma_+ & = \frac{\gamma_k^{-1}e^{\pi i \beta}|t|^{-\frac{1}{2}+\mu}}{1 - i\gamma_k^{-1}e^{\pi i \beta}|t|^{-\frac{1}{2}+\mu}}, \qquad \mu \in [0,\frac{1}{2}),\label{sigma+} \\
      \sigma_- & = \frac{\gamma_{k-1}e^{-\pi i \beta}|t|^{-\frac{1}{2}-\mu}}{1+i\gamma_{k-1}e^{-\pi i \beta}|t|^{-\frac{1}{2}-\mu}}, \qquad \mu \in [-\frac{1}{2},0).\label{sigma-}
  \end{align}
\end{proposition}

\begin{proof}
  When $\mu \in [0,\frac{1}{2})$ and $z \to 0$, it follows from \eqref{r1+}, \eqref{e+} and \eqref{q} that
  \begin{equation}\label{qat0_}
    Q(z) = R(z)\left(I + \frac{\gamma_k^{-1}|t|^{-\frac{1}{2}+\mu}}{f^{(0)}(z)}E^{(0)}(z)
    \begin{pmatrix} 0 & 1 \\ 0 & 0 \end{pmatrix}\left(E^{(0)}(z)\right)^{-1} \right).
  \end{equation}
  With \eqref{pinfinity_} and \eqref{e0_}, we have the local behavior near $z = 0$ for the analytic factor $E^{(0)}(z)$ as
  \begin{equation}
    E^{(0)}(z) = T \left( \frac{1}{\sqrt{2}}\begin{pmatrix} 1 & -1 \\ 1 & 1 \end{pmatrix} + \frac{iz}{2 \sqrt{2}}\begin{pmatrix} -1 & -1 \\ 1 & -1 \end{pmatrix} + O(z^2)\right)
    \left( e^{\frac{\pi i}{2}\beta }(1 + i \beta z + O(z^2)) \right)^{\sigma_3}.
  \end{equation}
  Substituting into \eqref{qat0_}, we have
  \begin{small}
    \begin{equation}
      \begin{aligned}
        Q(z)
        & = R(z)T
         \left(I+\frac{2\gamma_k^{-1}e^{\pi i \beta}|t|^{-\frac{1}{2}+\mu}}{z}
        (1+2i\beta z +O(z^2))
        \left(\frac{1}{2}\begin{pmatrix} -1 & 1 \\ -1 & 1 \end{pmatrix} + \frac{iz}{2}
        \begin{pmatrix} 0 & -1 \\ -1 & 0 \end{pmatrix} \right. \right.\\
        &\quad \left. \left.+ O(z^2)\right)\right)T^{-1}   \\
        & = \hat{Q}(z)\left(I + \frac{\gamma_k^{-1}e^{\pi i \beta}|t|^{-\frac{1}{2}+\mu}}{1 - i\gamma_k^{-1}e^{\pi i \beta}|t|^{-\frac{1}{2}+\mu}} \frac{1}{z}
        \begin{pmatrix} -1 & 1 \\ -1 & 1 \end{pmatrix} + O(z)\right)T^{-1}.
      \end{aligned}
    \end{equation}
  \end{small}
Here, we make use of the fact that $R(z)$ is actually analytic at $z = 0$. When $\mu \in [-\frac{1}{2},0)$, a similar computation yields \eqref{qat0}. 
 
 This completes our proof.
\end{proof}



\subsection{Final transformation}
Now, all the jumps matrices of $Q(z)$ tend to identity matrices, but there exists an isolated singularity at $z = 0$. Therefore, we introduce the final transformation as
\begin{equation}\label{ql}
  Q(z) = \mathcal{L}(z) + \frac{B^\pm \mathcal{L}(z)}{z}, \qquad z \in \mathbb{C} \setminus \Sigma_{R}.
\end{equation}
Here $B^\pm$ is a constant matrix with respect to $z$, and the choice of $``+"$ or $``-"$ depends on the range of $\mu$. Then, we obtain the following RH problem for $\mathcal{L}(z)$,
\begin{rhp}
  $\mathcal{L}(z)$ defined in \eqref{ql} satisfies the following properties.
  \begin{itemize}
    \item [(a)] $\mathcal{L}(z)$ is defined and analytic for $z \in \mathbb{C} \setminus \Sigma_R$.
    \item [(b)] $\mathcal{L}(z)$ satisfies the same jump conditions for $Q(z)$.
    \item [(c)] $\mathcal{L}(z)$ is analytic at $z = 0$.
    \item [(d)] As $z \to \infty$, we have
    $
      \mathcal{L}(z) = I + O(z^{-1}).
    $
  \end{itemize}
\end{rhp}

The explicit expression of $B^{\pm}$ can be obtained from condition $(b)$:
\begin{small}
\begin{align}
  B^+ & = \sigma_+\mathcal{L}(0)T \begin{pmatrix}-1 & 1 \\ -1 & 1\end{pmatrix}T^{-1}
  \left(\mathcal{L}(0)-\sigma_+\mathcal{L}'(0)T\begin{pmatrix}-1 & 1 \\ -1 & 1\end{pmatrix}T^{-1}\right)^{-1},\label{b+}
  \qquad \mu \in [0,\frac{1}{2}), \\
  B^- & = \sigma_-\mathcal{L}(0)T \begin{pmatrix}-1 & -1 \\ 1 & 1\end{pmatrix}T^{-1}
  \left(\mathcal{L}(0)-\sigma_-\mathcal{L}'(0)T\begin{pmatrix}-1 & -1 \\ 1 & 1\end{pmatrix}T^{-1}\right)^{-1},\label{b-}
  \quad \mu \in [-\frac{1}{2},0).
\end{align}
\end{small}
Moreover, since $\mathcal{L}(z)$ and $Q(z)$ share the same jump conditions, applying \eqref{matchforp-1_}, \eqref{matchforp1_},  \eqref{matchforp0_}, \eqref{jr!}, \eqref{e+} and \eqref{e-}, we have
\begin{equation}
  J_\mathcal{L}(z) =
  \begin{cases}
    I + O(t^{-\frac{1}{2}}), & \qquad |z| = \delta, \\
    I + O(t^{-1}), & \qquad |z \pm 1| = \delta,
  \end{cases}
\end{equation}
as $it \to +\infty$.

Therefore, a standard small norm analysis for $\mathcal{L}(z)$ leads us to
\begin{equation}\label{l_}
  \mathcal{L}(z) = I + \frac{\mathcal{L}_1(z)}{|t|^{\frac{1}{2}}} + o(t^{-\frac{1}{2}}).
\end{equation}
For our work, the explicit expression of $\mathcal{L}_1(z)$ is not needed. As $it \to +\infty$, substituting \eqref{l_} into \eqref{b+} and \eqref{b-} gives
\begin{align}
    B^+ & = \sigma_+ T \begin{pmatrix}-1 & 1 \\ -1 & 1\end{pmatrix}T^{-1}
    \left(I+O(t^{-\frac{1}{2}})\right), \qquad \mu \in [0,\frac{1}{2}),\label{b+asstoinfinity}   \\
    B^- & = \sigma_- T \begin{pmatrix}-1 & -1 \\ 1 & 1\end{pmatrix}T^{-1}
    \left(I+O(t^{-\frac{1}{2}})\right), \qquad \mu \in [-\frac{1}{2},0).\label{b-asstoinfinity}
\end{align}

\subsection{Asymptotics of $\Psi$ as $it \to +\infty$}

Recalling \eqref{integralforfred}, we require $(\Psi_1(t))_{11}$ to derive the large $t$ asymptotics of the Fredholm determinant.

Now tracing back the transformations $\Psi \mapsto A \mapsto R \mapsto Q \mapsto \mathcal{L}$ in \eqref{normalization}, \eqref{r!}, \eqref{q} and \eqref{ql}, and letting $z \to \infty$, we have
\begin{equation}
\begin{aligned}
  \Psi(z)
   & = e^{-\frac{\pi i}{2}\chi \sigma_3} \left(\mathcal{L}(z) + \frac{B^\pm \mathcal{L}(z)}{z} \right)\cdot e^{\frac{\pi i}{2}(\mu-\beta)\sigma_3}2^{\beta \sigma_3}
    \left(I + \frac{1}{z} \begin{pmatrix} -i\alpha & -\frac{i}{2} \\ \frac{i}{2} & i\alpha \end{pmatrix} + O(z^{-2})\right) \\
   &\quad \times 2^{-\beta \sigma_3}e^{\frac{\pi i}{2}\beta \sigma_3}z^{-\beta \sigma_3}D(z)^{\mu \sigma_3}
    D(z)^{-\chi \sigma_3} e^{\frac{1}{4}tz\sigma_3}\left(I - \frac{t}{8z}\sigma_3 + O(z^{-2}) \right).
\end{aligned}
\end{equation}
As $it \to +\infty$, comparing the above equation with \eqref{infinitybehaviorforpsi} gives us
\begin{equation}\label{psi111}
    (\Psi_1(t))_{11}  = \begin{cases}
    - \frac{t}{8} -i\alpha -ik -\sigma_+ + O(t^{-\frac{1}{2}}), \quad & \mu \in [0,\frac{1}{2})\\
    - \frac{t}{8} -i\alpha -ik -\sigma_- + O(t^{-\frac{1}{2}}), \quad & \mu \in [-\frac{1}{2},0),
    \end{cases}
\end{equation}
where $\sigma_{\pm}$ are given in \eqref{sigma+} and \eqref{sigma-}, respectively.


\section{Proof of the main results}\label{proof}

In this section, we complete our proof of Theorem \ref{mainresult_}. First, we will introduce two lemmas and one theorem that will assist us in deriving the final asymptotics.

\subsection{Preliminary lemmas and theorem}

Note that the large-$t$ asymptotics of $\Psi_1(t)$ contains the terms related to the parameter $k$. Thus for $\chi \geq 0$, we introduce the following notations
\begin{equation}\label{sk}
  |t_k'|=2\left(\nu+(\alpha+k)\ln|t|\right), \qquad |t_k|=2\left(\nu+(\alpha+k-\frac{1}{2})\ln|t|\right),
\end{equation}
\begin{equation}\label{xy}
  x_k= - i \gamma_k^{-1}e^{\pi i \beta}, \qquad y_k = i \gamma_{k-1}e^{-\pi i \beta}.
\end{equation}
Here $t, t_k', t_k$ are all negative pure imaginary numbers.

Now we introduce two lemmas.

\begin{lemma}
  For $|t| \in [|t_k'|,|t_{k+1}|]$, i.e. $\mu \in [0,\frac{1}{2}]$, we have
  \begin{equation}\label{lemma+}
    \begin{aligned}
    \int_{t_k'}^t (\Psi_1(\tau))_{11} d\tau =
    & -\frac{t^2-t_k'^2}{16} - i\alpha(t-t_k')
    -2k \ln(|t|^{-k-\alpha}e^{\frac{1}{2}|t|-\nu}) \\
    & \qquad \qquad - 2 \ln(1+x_k |t|^{-\frac{1}{2}-k-\alpha}e^{\frac{1}{2}|t|-\nu})+O(t^{-\frac{1}{2}}\ln t).
    \end{aligned}
  \end{equation}
\end{lemma}

\begin{proof}
  According to \eqref{sigma+} and \eqref{psi111}, a direct integration leads us to
  \begin{small}
  \begin{multline}
      \int_{t_k'}^t (\Psi_1(\tau))_{11}d\tau = -\frac{t^2-t_k'^2}{16} - i\alpha(t-t_k') - ik(t-t_k') \\
      - i \int_{t_k'}^t
      \frac{x_k |\tau|^{-\frac{1}{2}-k-\alpha}e^{\frac{1}{2}|\tau|-\nu}}{1+x_k |\tau|^{-\frac{1}{2}-k-\alpha}e^{\frac{1}{2}|\tau|-\nu}} d\tau  + O(t^{-\frac{1}{2}}\ln t).
    \end{multline}
  \end{small}
  With \eqref{sk}, we have
  \begin{equation}
    t-t_k' = -2i\ln(|t|^{-k-\alpha}e^{\frac{1}{2}|t|-\nu}).
  \end{equation}
  For the integration term, a change of variable yields
  \begin{equation}
    \begin{aligned}
      &  \int_{t_k'}^t
      \frac{x_k |\tau|^{-\frac{1}{2}-k-\alpha}e^{\frac{1}{2}|\tau|-\nu}}{1+x_k |\tau|^{-\frac{1}{2}-k-\alpha}e^{\frac{1}{2}|\tau|-\nu}} d\tau \\
      & = -2i \ln (1+x_k \tau^{-\frac{1}{2}-k-\alpha}e^{\frac{1}{2}\tau-\nu}) \Big{|}_{\tau = |t_k'|}^{|t|}\\
      &\qquad \qquad
      -2i(k+\alpha-\frac{1}{2})\int_{|t_k'|}^{|t|}
      \frac{x_k \tau^{-\frac{1}{2}-k-\alpha}e^{\frac{1}{2}\tau-\nu}}{1+x_k \tau^{-\frac{1}{2}-k-\alpha}e^{\frac{1}{2}\tau-\nu}} \frac{d\tau}{\tau} \\
      & = -2i \ln(1+x_k |t|^{-\frac{1}{2}-k-\alpha}e^{\frac{1}{2}|t|-\nu}) + O(t^{-1}\ln t).
    \end{aligned}
  \end{equation}
  This completes the proof of this lemma.
\end{proof}


Similarly, following the same deviation, we have
\begin{lemma}
  For $|t| \in [|t_k|,|t_k'|]$, $k \geq 1$, i.e. $\mu \in [-\frac{1}{2},0]$, we have
  \begin{equation}\label{lemma-}
   \begin{aligned}
    \int_{t_k}^t (\Psi_1(\tau))_{11}d\tau =
    & -\frac{t^2-t_k^2}{16} - i\alpha(t-t_k) - 2k \ln(|t|^{\frac{1}{2}-k-\alpha}e^{\frac{1}{2}|t|-\nu}) \\
    &  -2 \ln(1+y_k |t|^{-\frac{1}{2}+k+\alpha}e^{-\frac{1}{2}|t|+\nu}) + 2\ln(1+y_k)+O(t^{-\frac{1}{2}}\ln t).
   \end{aligned}
  \end{equation}
\end{lemma}

To prove the main result, we also need the following theorem, which improves  the result obtained in \cite[Theorem 4.]{Xu:Zhao2020}. 
\begin{theorem}\label{thmapp}
  For $\alpha > -\frac{1}{2}$, $\beta \in i \mathbb{R}$, and  $\gamma \to 1$ together with $it \to +\infty$, we have
  \begin{multline}\label{integrals0}
    \int_0^t \left(-\frac{1}{2}(\Psi_1(\tau))_{11} -\frac{\alpha^2-\beta^2}{\tau} \right) d\tau = \frac{t^2}{32} + \frac{i \alpha t}{2} - \left(\alpha^2 - \beta^2 + \frac{1}{4}\right)\ln\frac{|t|}{4} \\ +
    \ln \frac{\sqrt{\pi}G(\frac{1}{2})G(1+2\alpha)}{2^{2\alpha^2}G(1+\alpha+\beta)G(1+\alpha-\beta)} + O(t^{-\frac{1}{2}}\ln t),
  \end{multline}
  uniformly for
  \begin{equation}
    -\ln (1-\gamma) \geq \frac{1}{2}|t| -\alpha \ln |t|.
  \end{equation}
 Here, $G(\cdot)$ is the Barnes G-function.
\end{theorem}

\begin{proof}

To see how the speed of $\gamma$ tending to 1 affects our result more clearly, we set
\begin{equation}\label{gammato1}
  -\ln (1-\gamma) = \frac{1}{2}|t| -(\alpha - \epsilon)\ln |t|,
\end{equation}
where $\epsilon \geq 0$ is a constant.

We perform a nonlinear steepest descent analysis to the model RH problem. The procedure is similar to that in Section \ref{analysistopsi}. Here we adjust all the steps in Section \ref{analysistopsi} by setting $\chi = k = \mu = 0$. Under this circumstance, the $g$-function is  as simple as
\begin{equation}\label{gfunctionapp}
  g(z)|_{\chi=0} = \frac{1}{4}\sqrt{z^2-1}, \qquad z \in \mathbb{C} \setminus [-1,1].
\end{equation}
Hence the normalization transformation is defined as
\begin{equation}
  A(z) = \Psi(z;s)e^{-t(g(z)|_{\chi=0})\sigma_3}.
\end{equation}

Following the analysis we carried out in Section \ref{analysistopsi}, we construct the global parametrix $P^{(\infty)}(z)$ for $A(z)$ to be \eqref{pinfinity_} with $\chi = \mu = k = 0$, and the local parametrices near $z = \pm 1$ to be \eqref{p-1_} and \eqref{p1_} with $\chi = \mu = k = 0$, respectively. The construction of the local parametrix near the origin is slightly different. More precisely, we replace the $H(\zeta)$ in \eqref{p0_} with the following
\begin{equation}\label{hermiteparametrixapp}
  H(\zeta)|_{\chi=0} =
  \begin{pmatrix}
    1 & \frac{1}{2 \pi i} \int_{\mathbb{R}} \frac{w(z)}{z - \zeta}dz \\
    0 & 1
  \end{pmatrix}e^{-\frac{t}{2}\zeta^2}, \qquad \zeta \in \mathbb{C} \setminus \mathbb{R}.
\end{equation}
Again, we set $\zeta = |t|^{\frac{1}{2}}f^{(0)}(z)$, where the conformal mapping $f^{(0)}(z)$ is still the same as \eqref{conformalmappingat0_}. As $it \to +\infty$, we have
\begin{equation}\label{hinfinityapp}
  H(|t|^{\frac{1}{2}}f^{(0)}(z))|_{\chi=0} = \left(I + \frac{\gamma_0^{-1}}{|t|^{\frac{1}{2}}f^{(0)}(z)}\begin{pmatrix} 0 & 1 \\ 0 & 0\end{pmatrix}
  + O(t^{-1})\right)e^{-\frac{|t|}{2}(f^{(0)}(z))^2}.
\end{equation}

Now we are ready to introduce the final transformation as
\begin{equation}
 R(z)|_{\chi=0}=A(z)
 \begin{cases}
   \left(P^{(0)}(z)|_{\chi=0}\right)^{-1}, \qquad      & z \in U(0,\delta), \\
   \left(P^{(1)}(z)|_{\chi=0}\right)^{-1}, \qquad      & z \in U(1,\delta),  \\
   \left(P^{(-1)}(z)|_{\chi=0}\right)^{-1}, \qquad     & z \in U(-1,\delta), \\
   \left(P^{(\infty)}(z)|_{\chi=0}\right)^{-1}, \qquad & \mathrm{otherwise}. \\
 \end{cases}
\end{equation}
It is straightforward to see that the jump matrix of $R(z)|_{\chi=0}$ is given as
\begin{equation}
 J_{R}(z)|_{\chi=0}=
 \begin{cases}
     P^{(1)}(z)|_{\chi=0}\left(P^{(\infty)}(z)|_{\chi=0}\right)^{-1}, \qquad & z \in \partial U(1,\delta),    \\
     P^{(-1)}(z)|_{\chi=0}\left(P^{(\infty)}(z)|_{\chi=0}\right)^{-1}, \qquad & z \in \partial U(-1,\delta),    \\
     P^{(0)}(z)|_{\chi=0}\left(P^{(\infty)}(z)|_{\chi=0}\right)^{-1}, \qquad & z \in \partial U(0,\delta), \\
     P^{(\infty)}(z)|_{\chi=0}J_A(z)\left(P^{(\infty)}(z)|_{\chi=0}\right)^{-1}, \qquad & \mathrm{otherwise}.
 \end{cases}
\end{equation}
Here all the three small circles are clockwise-oriented.

According to the previous analysis, it is straightforward to see $J_{R}(z)|_{\chi=0} = I + O(t^{-1})$ on $\Sigma_R \setminus \{\partial U(0,\delta)\}$ as $it \to \infty$. While $z \in \partial U(0,\delta)$, with \eqref{hinfinityapp} we have
 \begin{multline}\label{matchforp0app}
     P^{(0)}(z)|_{\chi=0}\left( P^{(\infty)}(z)|_{\chi=0} \right)^{-1}\\
   =  I + \frac{\gamma_0^{-1}|t|^{-1-\epsilon}}{f^{(0)}(z)}E^{(0)}(z)|_{\chi=0}
    \begin{pmatrix} 0 & 1 \\ 0 & 0\end{pmatrix}(E^{(0)}(z)|_{\chi=0})^{-1} + O(t^{-1-\epsilon}). 
 \end{multline}
Now one can conclude that $J_{R}(z)$ tends to identity matrix uniformly for $z \in \Sigma_R$ as $it \to +\infty$. Thus, further transformations $R(z) \mapsto Q(z) \mapsto \mathcal{L}(z)$ in \eqref{q} and \eqref{ql} are not necessary.

As $it \to +\infty$, a small norm analysis of $R(z)|_{\chi=0}$ gives
\begin{equation}
  R(z)|_{\chi=0} = I + R_1(z)|t|^{-\frac{1}{2}-\epsilon} + R_2(z)|t|^{-1} + R_3(z)|t|^{-1-\epsilon}+ O(|t|^{-\frac{3}{2}}).
\end{equation}
By a computation of residue of $J_{R}(z)$ near $\pm1$, we have
\begin{equation}
  R_2(z) = \frac{\sigma_3}{2z} + O(z^{-1}), \qquad z \to \infty.
\end{equation}
Therefore, tracing back to the transformations $\Psi \mapsto A(z) \mapsto R(z)|_{\chi=0}$, we have
\begin{equation}\label{psiapp}
  (\Psi_1(t))_{11} = -\frac{t}{8} - i\alpha + 2(\hat{R}_1)_{11}t^{-\frac{1}{2}-\epsilon} + \frac{1}{2t} + 2(\hat{R}_3)_{11}t^{-1-\epsilon} + O(t^{-\frac{3}{2}}),
  \qquad it \to +\infty,
\end{equation}
where $(\hat{R}_1)$ and $(\hat{R}_3)$ are the coefficients of large-$z$ asymptotics of $R_1(z)$ and $R_3(z)$, respectively. However, their exact expressions are not required to obtain our result.

Now we set
\begin{equation}
  |t_{-1}| = 2(\nu + (\alpha-1)\ln |t|), \qquad |t_0'| = 2(\nu + \alpha \ln |t|).
\end{equation}
Apparently, when $|t| \in (0, |t_{-1}|]$, it is easily seen from \eqref{gammato1} that $\epsilon \geq 1$. With \eqref{psiapp}, we have
\begin{equation}
  \begin{aligned}
    & \int_0^{t} \left(-\frac{1}{2}(\Psi_1(\tau))_{11} -\frac{\alpha^2-\beta^2}{\tau} \right) d\tau \\
    & = \frac{t^2}{32} + \frac{i\alpha t}{2} - \left(\alpha^2 -\beta^2 +\frac{1}{4}\right) \ln \frac{|t|}{4} + d_0(\gamma) + O(t^{-\frac{1}{2}} ), \qquad it \to +\infty, \gamma \to 1,
  \end{aligned}
\end{equation}
uniformly for
\begin{equation}
  -\ln (1-\gamma) \geq \frac{|t|}{2} -(\alpha-1) \ln |t|,
\end{equation}
where $d_0(\gamma)$ is a constant may still depend on $\gamma$. Meanwhile, when $|t| \in [|t_{-1}|, |t_0'|]$ which indicates $0 \leq \epsilon \leq 1$, with \eqref{psiapp}, we have
\begin{equation}\label{intofh}
  \begin{aligned}
    & \int_0^{t} \left(-\frac{1}{2}(\Psi_1(\tau))_{11} -\frac{\alpha^2-\beta^2}{\tau} \right) d\tau \\
    & = \int_{0}^{t_{-1}} \left(-\frac{1}{2}(\Psi_1(\tau))_{11} -\frac{\alpha^2-\beta^2}{\tau} \right)d\tau + \int_{t_{-1}}^{t} \left(-\frac{1}{2}(\Psi_1(\tau))_{11} -\frac{\alpha^2-\beta^2}{\tau} \right) d\tau \\
    & = \frac{t^2}{32} + \frac{i\alpha t}{2} - \left(\alpha^2 -\beta^2 +\frac{1}{4}\right) \ln \frac{|t|}{4} + d_0(\gamma) + O(t^{-\frac{1}{2}}\ln t ), \qquad it \to +\infty, \gamma \to 1,
  \end{aligned}
\end{equation}
uniformly for
\begin{equation}
   \frac{|t|}{2} - \alpha \ln |t| \leq -\ln (1-\gamma) \leq \frac{|t|}{2} -(\alpha-1) \ln |t|.
\end{equation}
Here we take advantage of the length of the second integral interval is of $O(\ln t)$. Furthermore, compared with \cite[Eq. (7.23)]{Xu:Zhao2020}, we conclude
\begin{equation}\label{d0}
    d_0(\gamma)  =
    \ln \frac{\sqrt{\pi}G(\frac{1}{2})G(1+2\alpha)}{2^{2\alpha^2}G(1+\alpha+\beta)G(1+\alpha-\beta)}(1+o(1)), \qquad \gamma \to 1.
\end{equation}

In order to give more details of the error term of $d_0(\gamma)$, let us consider the derivative of $(\Psi_1(t))_{11}$ with respect to $\gamma$. With \eqref{psiapp}, we have
\begin{equation}
  \partial \frac{(\Psi_1(t))_{11}}{\partial \gamma} = O(t^{-\frac{1}{2} - \epsilon}), \qquad it \to +\infty.
\end{equation}
Then, with the same derivation, we have
\begin{equation}
  \int_0^t \frac{\partial}{\partial \gamma} \left(-\frac{1}{2}(\Psi_1(\tau))_{11} -\frac{\alpha^2-\beta^2}{\tau} \right) d\tau = O(t^{-\frac{1}{2}}\ln t),
\end{equation}
uniformly for $-\ln(1-\gamma) \geq |t|/2 - \alpha \ln |t|$. Comparing this equation with \eqref{intofh} gives us
\begin{equation}
\frac{d}{d \gamma} d_0(\gamma) = O(t^{-\frac{1}{2}}\ln t).
\end{equation}
The above equation and \eqref{d0} yield
\begin{equation}
    d_0(\gamma)  =
    \ln \frac{\sqrt{\pi}G(\frac{1}{2})G(1+2\alpha)}{2^{2\alpha^2}G(1+\alpha+\beta)G(1+\alpha-\beta)}(1 + O(t^{-\frac{1}{2}}\ln t)), \qquad \gamma \to 1.
\end{equation}

Therefore, we complete the proof of Theorem \ref{thmapp}. 
\end{proof}

\subsection{Proof of Theorem \ref{mainresult_}}

Combined the two lemmas and Theorem \ref{thmapp}, we are ready to obtain our final result.

First, let us consider $|t| \in [|t_q'|,|t_{q+1}|)$, where $q \in \mathbb{N}_{\geq 0}$. It follows  from \eqref{lemma+}, \eqref{lemma-} and \eqref{integrals0} that
\begin{equation}\label{tqtq+1}
  \begin{aligned}
    & \int_0^t \left(-\frac{1}{2}(\Psi_1(\tau))_{11} -\frac{\alpha^2-\beta^2}{\tau} \right) d\tau \\
    & = \left(\int_0^{t_0'} +\int_{t_0'}^{t_1} +\int_{t_1}^{t_1'}  +\cdots +\int_{t_{q}'}^{t} \right) \left(-\frac{1}{2}(\Psi_1(\tau))_{11} -\frac{\alpha^2-\beta^2}{\tau} \right) d\tau    \\
    & = \frac{t^2}{32} + \frac{i\alpha t}{2}
    - \left(\alpha^2-\beta^2+\frac{1}{4}\right)\ln\left(\frac{|t|}{4}\right) +
    \ln \left(\frac{\sqrt{\pi}G^2(1/2)G(1+2\alpha)}{2^{2\alpha^2}G(1+\alpha+\beta)G(1+\alpha-\beta)}\right)   \\
    & \qquad + \sum_{k=0}^{q-1}k\ln(|t|^\frac{1}{2}) + \sum_{k=0}^{q-1}\ln(1+x_k) +
    \sum_{k=1}^{q}k\ln(|t|^\frac{1}{2}) - \sum_{k=1}^{q}\ln(1+y_k)  \\
    & \qquad + q \ln(|t|^{-q-\alpha}e^{\frac{1}{2}|t|-\nu}) + \ln(1+x_q |t|^{-\frac{1}{2}-q-\alpha}e^{\frac{1}{2}|t|-\nu}) + O(t^{-\frac{1}{2}}\ln t).
  \end{aligned}
\end{equation}
In case when $q=0$, we take all the sum $\sum (\cdots) \equiv 0$. Hence,
\begin{equation}
\begin{aligned}
  & \sum_{k=0}^{q-1}k\ln(|t|^\frac{1}{2}) + \sum_{k=0}^{q-1}\ln(1+x_k) +
  \sum_{k=1}^{q}k\ln(|t|^\frac{1}{2}) - \sum_{k=1}^{q}\ln(1+y_k)  \\
  & \qquad + q \ln(|t|^{-q-\alpha}e^{\frac{1}{2}|t|-\nu}) + \ln(1+x_q |t|^{-\frac{1}{2}-q-\alpha}e^{\frac{1}{2}|t|-\nu}) \\
  & = \sum_{k=0}^{q-1}\ln \left( x_k |t|^{\frac{1}{2}+k-q-\alpha}e^{\frac{1}{2}|t|-\nu}\right) + \ln(1+x_q |t|^{-\frac{1}{2}-q-\alpha}e^{\frac{1}{2}|t|-\nu})  \\
  & = \sum_{k=0}^{q}\ln \left(1 + x_k |t|^{-\frac{1}{2}-k-\alpha}e^{\frac{1}{2}|t|-\nu}\right) + O(t^{-\frac{1}{2}}).
\end{aligned}
\end{equation}
Here we apply the fact that $x_k y_{k+1} = 1$; see the definitions of $x_k$ and $y_k$ in \eqref{xy}. Substituting above equation into \eqref{tqtq+1}, with \eqref{integralforfred}, we complete the proof of \eqref{fdeterminant} for $|t| \in [|t_q'|,|t_{q+1}|)$ with $q \in \mathbb{N}_{\geq 0}$.

Similarly, when $|t| \in [|t_{q+1}|,|t_{q+1}'|]$, where $q \in \mathbb{N}_{\geq 0}$, with \eqref{lemma+}, \eqref{lemma-} and \eqref{integrals0}, we have
  \begin{equation}
    \begin{aligned}
      & \int_0^t \left(-\frac{1}{2}(\Psi_1(\tau))_{11} -\frac{\alpha^2-\beta^2}{\tau} \right) d\tau \\
      & = \frac{t^2}{32} + \frac{i\alpha t}{2}
      - \left(\alpha^2-\beta^2+\frac{1}{4}\right)\ln\left(\frac{|t|}{4}\right)
      + \ln \left(\frac{\sqrt{\pi}G^2(1/2)G(1+2\alpha)}{2^{2\alpha^2}G(1+\alpha+\beta)G(1+\alpha-\beta)}\right)   \\
      & \quad + \sum_{k=0}^{q}k\ln(|t|^\frac{1}{2}) + \sum_{k=0}^{q}\ln(1+x_k) +
      \sum_{k=1}^{q}k\ln(|t|^\frac{1}{2}) - \sum_{k=1}^{q}\ln(1+y_k)  \\
      & \quad + (q+1) \ln(|t|^{-\frac{1}{2}-q-\alpha}e^{\frac{1}{2}|t|-\nu}) +\ln(1+y_{q+1} |t|^{\frac{1}{2}+q+\alpha}e^{-\frac{1}{2}|t|+\nu})\\
      &\quad-\ln(1+y_{q+1})+O(t^{-\frac{1}{2}}\ln t).
    \end{aligned}
  \end{equation}
Furthermore, we have
\begin{equation}
  \begin{aligned}
    & \sum_{k=0}^{q}k\ln(|t|^\frac{1}{2}) + \sum_{k=0}^{q}\ln(1+x_k) +
    \sum_{k=1}^{q}k\ln(|t|^\frac{1}{2}) - \sum_{k=1}^{q}\ln(1+y_k)  \\
    & \qquad + (q+1) \ln(|t|^{-\frac{1}{2}-q-\alpha}e^{\frac{1}{2}|t|-\nu}) +\ln(1+y_{q+1} |t|^{\frac{1}{2}+q+\alpha}e^{-\frac{1}{2}|t|+\nu})-\ln(1+y_{q+1}) \\
    & = \sum_{k=0}^{q}\ln\left(x_k |t|^{-\frac{1}{2}+k-q-\alpha}e^{\frac{1}{2}|t|-\nu}\right) + \ln(1+y_{q+1} |t|^{\frac{1}{2}+q+\alpha}e^{-\frac{1}{2}|t|+\nu})    \\
    & = \sum_{k=0}^{q}\ln\left(1 + x_k |t|^{-\frac{1}{2}-k-\alpha}e^{\frac{1}{2}|t|-\nu}\right) - \ln \left(\frac{1 + x_q |t|^{-\frac{1}{2}-q-\alpha}e^{\frac{1}{2}|t|-\nu}}{ x_q |t|^{-\frac{1}{2}-q-\alpha}e^{\frac{1}{2}|t|-\nu}}\right) \\
    & \qquad \qquad + \ln(1+y_{q+1} |t|^{\frac{1}{2}+q+\alpha}e^{-\frac{1}{2}|t|+\nu}) + O(t^{-\frac{1}{2}}) \\
    & = \sum_{k=0}^{q}\ln\left(1 + x_k |t|^{-\frac{1}{2}-k-\alpha}e^{\frac{1}{2}|t|-\nu}\right) + O(t^{-\frac{1}{2}}).
  \end{aligned}
\end{equation}
Hence with \eqref{integralforfred}, we complete the proof of \eqref{fdeterminant} for $|t| \in [|t_q'|,|t_{q+1}'|]$. Note that, the lower constraint on $t$ is actually artificial. One can let $|t| \in [0,|t_{q+1}'|]$, then some factors in the sum would move to the error term. After adjusting the error term, we are still able to reproduce the same expression.

Then, with Theorem \ref{thmapp}, we complete the proof of Theorem \ref{mainresult_}. \qed

\section*{Acknowledgments}
Dan Dai was partially supported by a grant from the City University of Hong Kong (Project No. 7005597), and grants
from the Research Grants Council of the Hong Kong Special Administrative Region, China (Project No. CityU 11311622 and CityU 11306723). He also thanks the Okinawa Institute of Science and Technology (OIST) for the hospitality during the Theoretical Sciences Visiting Program (TSVP). Part of the preparation of this work was done during the period. Luming Yao was partially supported by National Natural Science Foundation of China under grant number 12271105.

\begin{appendices}

\section{Bessel parametrix}\label{bessel}

In this appendix, we introduce the well-known Bessel parametrix $\Phi^{B}(\zeta)$. Consider the following model RH problem.
\begin{rhp}\label{rhpforj}
  We look for a $2 \times 2$ matrix-valued function $\Phi^{B}(\zeta)$ with properties:
\begin{itemize}
  \item[(a)] $\Phi^{B}(\zeta)$ is analytic for $\zeta \in \mathbb{C}\setminus \{\Gamma_1^{(B)} \cup \Gamma_2^{(B)} \cup \Gamma_3^{(B)}\}$, where
  \begin{equation}
    \Gamma_1^{(B)} = e^{\frac{2}{3}\pi i}\mathbb{R_+}, \qquad
    \Gamma_2^{(B)} = (-\infty, 0], \qquad
    \Gamma_3^{(B)} = e^{-\frac{2}{3}\pi i}\mathbb{R_+}.
  \end{equation}
  \item[(b)] $\Phi^{B}(\zeta)$ satisfies the jump conditions
  \begin{equation}
    \Phi^{B}_+(\zeta) = \Phi^{B}_-(\zeta)
    \begin{cases}
       \begin{pmatrix} 1 & 0 \\ 1 & 1 \end{pmatrix}, \quad & \zeta \in \Gamma_1^{(B)},\\
       \begin{pmatrix} 0 & 1 \\ -1 & 0 \end{pmatrix}, \quad & \zeta \in \Gamma_2^{(B)},\\
       \begin{pmatrix} 1 & 0 \\ 1 & 1 \end{pmatrix}, \quad & \zeta \in \Gamma_3^{(B)}.\\
    \end{cases}
  \end{equation}
  \item[(c)] As $\zeta \to \infty$, we have
  \begin{equation}\label{besselinfinity}
    \Phi^{B}(\zeta)=\zeta^{-\frac{1}{4}\sigma_3}\frac{1}{\sqrt2}\begin{pmatrix} 1 & i \\ i & 1 \end{pmatrix}\left(I+\frac{1}{8\sqrt{\zeta}}\begin{pmatrix} -1 & -2i \\ -2i & 1 \end{pmatrix}+O\left(\frac{1}{\zeta}\right)\right)e^{\sqrt{\zeta}\sigma_3}.
  \end{equation}
  Here the branches of the term $\zeta^{\frac{1}{4}}$ and $\zeta^{\frac{1}{2}}$ are taken along $\zeta \in (-\infty, 0)$ such that $\arg \zeta \in (-\pi, \pi)$.
\end{itemize}
\end{rhp}

We first define a $2 \times 2$ matrix function as
\begin{equation}\label{jb_}
  \Phi_B(\zeta) = \pi^{\frac{1}{2}\sigma_3}\begin{pmatrix} I_0(\sqrt{\zeta}) & \frac{i}{\pi}K_0(\sqrt{\zeta}) \\ \pi i\sqrt{\zeta}I_0'(\sqrt{\zeta}) & -\sqrt{\zeta}K_0'(\sqrt{\zeta}) \end{pmatrix},
\end{equation}
where $I_0,K_0$ are the modified Bessel functions with branch cut $(-\infty,0]$ such that $\arg \zeta \in (-\pi,\pi)$. And it is easy to verify that
\begin{equation}
  \Phi_{B,+}(\zeta) = \Phi_{B,-}(\zeta)\begin{pmatrix} 1 & 1 \\ 0 & 1 \end{pmatrix}, \qquad
  \zeta \in (-\infty,0).
\end{equation}
Then, the well-known solution to the RH problem \ref{rhpforj} is given explicitly as
\begin{equation}\label{bessel parametrix}
 \Phi^{B}(\zeta) = \Phi_B(\zeta)
  \begin{cases}
   I,  & \arg \zeta \in (-\frac{2\pi}{3},\frac{2\pi}{3}), \\
   \begin{pmatrix} 1 & 0 \\ -1 & 1 \end{pmatrix},  & \arg \zeta \in (\frac{2\pi}{3},\pi), \\
   \begin{pmatrix} 1 & 0 \\ 1 & 1 \end{pmatrix}  & \arg \zeta \in (-\frac{2\pi}{3},-\pi),
  \end{cases}
\end{equation}
For the sake of convenience, let us also recall the following formula
\begin{equation}
\frac{d\Phi^{B}(\zeta)}{d\zeta}\Phi^{B}(\zeta)^{-1} =
\begin{pmatrix}
  0 & -\frac{i}{2\zeta} \\ \frac{i}{2} & 0
\end{pmatrix}.
\end{equation}

\end{appendices}

\end{document}